\documentclass[11pt,longtable,letterpaper,oneside,reqno]{amsart}
\usepackage{amsmath,amsfonts,amsthm,amssymb,amscd,mathtools,stmaryrd,url}
\usepackage{bbm}
\usepackage{epsfig}
\usepackage{graphicx}
\usepackage{latexsym}
\usepackage{longtable}
\usepackage{float}
\usepackage{hyperref}
\usepackage{color}

\DeclareFontFamily{OT1}{rsfs}{}
\DeclareFontShape{OT1}{rsfs}{n}{it}{<-> rsfs10}{}
\DeclareMathAlphabet{\mathscr}{OT1}{rsfs}{n}{it}

\newcommand{\mymod}[1]{(\operatorname{mod} #1)}

\renewcommand{\Re}{{\operatorname{Re}}}
\renewcommand{\Im}{{\operatorname{Im}}}

\newtheorem{prop}{Proposition}[section]
\newtheorem{thm}[prop]{Theorem}

\newtheorem{lem}[prop]{Lemma}

\newtheorem {defn }{Definition}

\newenvironment{rem}{{\bf Remark.}}{}
\numberwithin{equation}{section}
\usepackage[margin=1in]{geometry}
\usepackage{lipsum}

\begin{document}
\title{The fourth moment of quadratic Dirichlet $L$-functions}

\begin{abstract}
We study the fourth moment of quadratic Dirichlet $L$-functions at $s= \frac{1}{2}$. We show an asymptotic formula under the generalized Riemann hypothesis, and obtain a precise lower bound unconditionally. The proofs of these results  follow closely arguments of 
Soundararajan and Young
\cite{Sound-Young}  and Soundararajan \cite{Sound00}.
\end{abstract}
\author[Q. Shen]{Quanli Shen}
\address{Department of Mathematics and Computer Science, University of Lethbridge, 4401 University Drive, Lethbridge, Alberta, T1K 3M4 Canada}
\email{quanli.shen@uleth.ca}
\thanks{
Research for this article is partially supported by the  University of Lethbridge.
}
\date{\today}
\maketitle


%
\section{Introduction}
Let $\chi_d = \left(\frac{d}{\cdot} \right)$ be a real primitive Dirichlet  character modulo $d$ given by the Kronecker symbol, where $d$ is a fundamental discriminant. The $k$-th moment of quadratic Dirichlet $L$-functions is   
\begin{equation}
 \sideset{}{^\flat}\sum_{0<d\leq X}L(\tfrac{1}{2},\chi_d)^k,
 \label{equ:k-th}
\end{equation}
where $\sideset{}{^\flat}\sum$  denotes the sum over fundamental discriminants, and $k$ is  a positive real number. One great motivation to study  \eqref{equ:k-th} comes from Chowla's conjecture, which states that $L(\tfrac{1}{2},\chi_d) \neq 0$ for all fundamental discriminants $d$.
The current best result toward this conjecture is Soundararajan's celebrated work \cite{Sound00}  in 2000, where it was proven that $L(\frac{1}{2},\chi_{8d}) \neq 0$ for at least $87.5 \%$ of the odd square-free integers $d\geq 0$. The key to the proof  is the evaluation of mollified first and second moments of quadratic Dirichlet $L$-functions.

 In 2000, using a random matrix model, Keating and Snaith \cite{Keating-Snaith02} conjectured  that for any  positive real number  $k$,
\begin{align}
  \sideset{}{^\flat}\sum_{|d|\leq X}L(\tfrac{1}{2},\chi_d)^k \sim C_kX(\log X)^{\frac{k(k+1)}{2}},
  \label{conj0}
\end{align}
where $C_k$ are explicit constants. 
Various researchers have studied versions of these moments summed over certain subsets of the fundamental discriminants. 
For instance, in (\ref{equ:k-th})  we consider positive fundamental discriminants.   However, there are no difficulties in also 
studying negative fundamental discriminants.
Some articles even consider characters of the  form $\chi_{8d}$, where $d$ are odd positive square-free integers. 
The main reason researchers study these special cases, rather than consider all fundamental discriminants, is to focus 
on the methods and techniques.  It is possible to establish results for all fundamental discriminants, but this 
would involve more  cases that need to be studied. 
The conjecture analogous to  (\ref{conj0}) for characters of the form $\chi_{8d}$, which can be established by using Keating and Snaith's method  \cite{Keating-Snaith02}, was obtained in  Andrade and Keating's paper \cite[Conjecture 2]{Andrade-Keating01}. For any  positive  real number $k $, it was conjectured that 
\begin{align}
 \sideset{}{^*}\sum_{\substack{0 < d \leq X \\ (d,2)=1} }L\left(\tfrac{1}{2},\chi_{8d} \right)^k \sim \frac{4a_{k}}{\pi^2}\frac{G(k+1)\sqrt{\Gamma(k+1)}}{\sqrt{G(2k+1)\Gamma(2k+1)}} X (\log X)^{\frac{k(k+1)}{2}},
\label{conj}
\end{align}
where 
$ \sideset{}{^*}\sum$ denotes the sum over square-free integers,
 $G(z)$ is the Barnes $G$-function, and 
 \begin{equation}
a_{k} := 2^{-\frac{k(k+2)}{2}} 
\prod_{(p,2)=1} \frac{(1-\frac{1}{p})^{\frac{k(k+1)}{2}}}{1+\frac{1}{p}} 
\left( 
\frac{
 (1+\frac{1}{\sqrt{p}} )^{-k} +
 (1-\frac{1}{\sqrt{p}} )^{-k}
 }{2}
 +\frac{1}{p}
\right).
\label{equ:a-k}
\end{equation}

In this paper, we  prove  the conjecture in (\ref{conj}) for $k=4$ assuming  the generalized Riemann hypothesis (GRH).
\begin{thm}
Assume GRH for $L(s,\chi_d)$ for all fundamental discriminants $d$.
For any $\varepsilon >0$, we have 
\[
 \sideset{}{^*}\sum_{\substack{0<d\leq X\\(d,2)=1}}L(\tfrac{1}{2},\chi_{8d})^4= \frac{a_{4 }}{2^6 \cdot 3^3 \cdot 5^2 \cdot 7 \cdot \pi^2}X{(\log X)}^{10} + O \left(X (\log X)^{9.75+\varepsilon} \right).
\]
\label{main-thm}
\end{thm}
The proof of Theorem \ref{main-thm} largely follows Soundararajan and Young's paper \cite{Sound-Young} in 2010 and Soundararajan's paper \cite{Sound00}  in 2000. In \cite{Sound-Young},  Soundararajan and Young
proved an asymptotic formula for the second moment of quadratic twists of a modular $L$-function, obtaining the leading main term.
Experts believed  that the methods and techniques in  \cite{Sound-Young} could be used to evaluate the fourth moment of quadratic Dirichlet $L$-functions.  Motivated by this expectation, we established Theorem \ref{main-thm}.  In fact, Theorem \ref{main-thm} may be viewed as a version of \cite[Theorem 1.2]{Sound-Young} where $f$ is an Eisenstein series. The main difference between 
this article and  \cite{Sound-Young} is that 
the off-diagonal terms (see just after \eqref{equ:13.a} for a precise definition) contribute to the main term, whereas in 
 \cite{Sound-Young} they are part of the error term.   
 We use  techniques from \cite[Sections 5.2, 5.3]{Sound00} to evaluate the off-diagonal terms and this is main new input.
 These terms may be written as a certain multiple complex integral.
 One of the difficulties in evaluating this integral is that the integrand has high order poles and this makes the calculation more  intricate.
It should be noted that  in 2017 Florea \cite{Florea02}  has proven an asymptotic formula for the fourth moment of quadratic Dirichlet $L$-functions in the function field setting,  with extra lower main terms. 

Similar to \cite[Theorem 1.1]{Sound-Young}, we  obtain an unconditional lower bound that matches the conjectured asymptotic formula (\ref{conj}).  This result was  stated  without proof by Rudnick and Soundararajan \cite{Rudnick-Soundararajan} in 2006.
\begin{thm}
Unconditionally,
we have
\[
 \sideset{}{^*}\sum_{\substack{0<d\leq X\\(d,2)=1}}L(\tfrac{1}{2},\chi_{8d})^4 \geq \left(\frac{a_{4}}{2^6 \cdot 3^3 \cdot 5^2 \cdot 7 \cdot \pi^2} +o(1) \right)X(\log X)^{10}.
\]
\label{main-thm1}
\end{thm}
We now introduce more refined conjectures  for the moments of quadratic Dirichlet $L$-functions   and 
provide a brief history of  related results.  
In 2005, Conrey, Farmer, Keating, Rubinstein and Snaith \cite{Conrey-Farmer-Keating-Rubinstein-Snaith}   gave a more precise  conjecture, including all other principal lower order terms, 
%
%
%
\begin{equation}
\label{equ:conj-cfkrs0}
\sideset{}{^\flat}\sum_{0<d\leq X}L(\tfrac{1}{2},\chi_d)^k  = XP_{\frac{k(k+1)}{2}}(\log X) + E_k(X),
\end{equation}
where  $k$ is a positive integer, $P_n(x)$ is an explicit polynomial of degree $n$, and $E_k(X)= o_k (X)$.
For characters of the form $\chi_{8d}$, their conjecture may be written as  
\begin{equation}
\sideset{}{^*}\sum_{\substack{0<d\leq X \\ (d,2) =1 }}L(\tfrac{1}{2},\chi_{8d})^k  = XQ_{\frac{k(k+1)}{2}}(\log X) + \hat{E}_k(X),
\label{equ:conj-cfkrs}
\end{equation}
where  $ Q_n(x)$ is another  explicit  polynomial of degree $n$, and $\hat{E}_k(X) = o_k (X)$.

In 1981, Jutila \cite{Jutila02} established  \eqref{equ:conj-cfkrs0} for $k=1$ with  $E_1(X)= O(X^{\frac{3}{4}+\varepsilon})$.
In 1985,  Goldfeld and Hoffstein \cite{Goldfeld-Hoffstein}  improved this  to $E_1(X)= O(X^{\frac{19}{32}+\varepsilon})$
by using multiple Dirichlet series. 
Their work implies the error $O(X^{\frac{1}{2}+\varepsilon})$ for a smoothed version of the sum  in \eqref{equ:conj-cfkrs0}
when $k=1$.  This  was later   obtained by Young \cite{Young01} in 2009, using a different technique 
based  on a recursive method and a study of shifted moments. We remark that  Alderson and Rubinstein \cite{Alderson-Rubinstein}  conjectured  that $E_1(X) = O(X^{\frac{1}{4}+\varepsilon})$.  In 1981, the second moment was established  by Jutila  \cite{Jutila02},
\begin{align*}
\sideset{}{^\flat}\sum_{|d|\leq X}L(\tfrac{1}{2},\chi_d)^2 = C_2X(\log X)^3 + O \left(X(\log X)^{\frac{5}{2}+\varepsilon} \right).
\end{align*}
In 2000, Soundararajan \cite{Sound00} improved this by obtaining the full main term in \eqref{equ:conj-cfkrs}, in the case $k=2$,
with the  power savings $\hat{E}_2(X) = O (X^{\frac{5}{6}+\varepsilon})$. In 2020, Sono \cite{Sono} improved this to $ O (X^{\frac{1}{2}+\varepsilon})$ for a smoothed variant of $\hat{E}_2(X)$. In  \cite{Sound00}
Soundararajan  was the first to prove an asymptotic for the
 third moment, obtaining $\hat{E}_3(X) =O(X^{\frac{11}{12}+\varepsilon})$.
 In 2003,  Diaconu, Goldfeld and Hoffstein \cite{Diaconu-Goldfeld-Hoffstein}   improved this to 
   $E_3(X) = O(X^{0.85 \dots +\varepsilon})$   by using  multiple Dirichlet series techniques.
In 2013, Young \cite{Young} further improved this to $  O(X^{\frac{3}{4}+\varepsilon})$ for  a  smoothed version of $\hat{E}_3(X)$
by using similar techniques to \cite{Young01}. Recently, in 2018, Diaconu and Whitehead \cite{Diaconu-Whitehead} 
improved Young's result by showing that a smoothed version of $\hat{E}_3(X)$  is of size $cX^{\frac{3}{4}}+O(X^{\frac{2}{3} +\varepsilon})$, for some $c \in \mathbb{R}$.   This verified a conjecture of Diaconu, Goldfeld and Hoffstein  \cite{Diaconu-Goldfeld-Hoffstein}
of the presence of a secondary lower order term.  Zhang   \cite{Zhangqiao} had previously  conditionally established a secondary term of size $X^{\frac{3}{4}}$ in 2005.

For the family of quadratic Dirichlet $L$-functions, moments higher than four  have not been asymptotically evaluated. This seems  beyond  current  techniques. However, there are celebrated results on upper and lower bounds of the moments.
In 2006, Rudnick and Soundararajan \cite{Rudnick-Soundararajan} proved  the lower bound 
\[
\sideset{}{^\flat}\sum_{0<d\leq X} L(\tfrac{1}{2},\chi_d)^k \gg_k X(\log X)^{\frac{k(k+1)}{2}}
\]
for all even natural number $k \geq 1$. 
 In 2009,  Soundararajan \cite{Sound01}  proved under GRH that  for all positive real $k$,
 \begin{equation}
\sideset{}{^\flat}\sum_{0<d\leq X}L(\tfrac{1}{2},\chi_d)^k\ll_{k,\varepsilon} X(\log X)^{\frac{k(k+1)}{2}+\varepsilon}.
\label{Sound-upper}
\end{equation}
In 2013,  Harper \cite{Harper}, assuming GRH, improved this to 
 \[
\sideset{}{^\flat}\sum_{0<d\leq X}L(\tfrac{1}{2},\chi_d)^k\ll_{k} X(\log X)^{\frac{k(k+1)}{2}}.
\] 

The method of this paper is largely based on the arguments and techniques in \cite{Sound-Young}  and \cite{Sound00}.  We use the approximate functional equation for  Dirichlet $L$-functions, and then employ the Poisson summation formula to separate the summation into  diagonal terms, off-diagonal terms, and  error terms.  Both  diagonal and off-diagonal terms  contribute to the main term.  
To bound the error terms, by following the argument in \cite{Sound01,Sound-Young}, under GRH, we established  an upper bound for the shifted moments of quadratic Dirichlet $L$-functions (see Theorem \ref{thm:3.1}).   

With further effort, one might be able to heuristically  obtain all the main terms that are expected from 
the conjecture of Conrey et al. in  (\ref{equ:conj-cfkrs}).  However, the computation will be complicated. 
It might be simplified by considering a shifted version of the fourth moment, analogous to the calculation in \cite{Young01}.
Florea considered the function field version of the fourth moment in  \cite{Florea02}. In her work she was able to
identify all the main terms as given by a conjecture of Andrade-Keating 
\cite[Conjecture 5]{Andrade-Keating01} (the function field analogue of  (\ref{equ:conj-cfkrs})). 
By using a  recursive method,  Florea  obtained extra lower main terms in this case.
 It is possible that  her techniques may be employed to obtain additional lower main terms in Theorem \ref{main-thm}
and we hope to revisit this in future work. 
However,  one would need to apply the approximate functional equation for  the fourth power of the $L$-function rather than the second power (\ref{equ:2.1.1}).  In addition,  one would have to 
eliminate the use of  the parameters $U_1,U_2$ in   (\ref{equ:13.1+}).  
In our article, we use the  approximate functional equation for  the second power of the $L$-function 
as it is  necessary to obtain the unconditional lower bound in Theorem \ref{main-thm1}.

The outline of this paper is as follows. The proof of Theorem \ref{main-thm} and \ref{main-thm1}  proceed simultaneously. In Section \ref{sec:tool},  we introduce some tools. In Section \ref{sec:setup}, we set up the evaluation of the  fourth moment.
We apply the Poisson summation formula to split the fourth moment into diagonal, off-diagonal, and error terms. 
 We evaluate the  diagonal terms and off-diagonal terms in Section \ref{sec:k=0} and Section \ref{sec:k=square}, respectively. The error terms are bounded in Section \ref{sec:error}. The proofs of Theorem \ref{main-thm} and \ref{main-thm1}  are completed in Section \ref{sec:proof-main}. Finally, we give the  proof of Theorem \ref{thm:3.1}  in Section \ref{sec:upper bd}. 
 \\

\noindent{\textbf{Notation.}}
In this paper, we shall use  the convention that $\varepsilon>0$ denotes an arbitrary small constant which may vary in different situations. For two functions $f(x)$ and $g(x)$, we shall use the notation $f(x) = O(g(x))$, $f(x) \ll g(x)$ to mean there exists  a  constant $C$ such  that $|f(x)| \leq  C|g(x)|$  for all sufficiently large $x$.   If we write $f(x) = O_a(g(x))$ or $f(x) \ll_a g(x)$, then we mean that the corresponding constants depend on $a$. Throughout the  paper, the big $O$ may depend on $\varepsilon$.

\section{Basic tools} \label{sec:tool}
In this section, we introduce several tools that shall be used in this article.
\subsection{Approximate functional equation}
For $\xi > 0$, define
\begin{align}
\omega(\xi)&:=\frac{1}{2\pi i}\int_{(c)} \pi^s g(s)  \xi^{-s}\frac{ds}{s}, \ c>0,
\label{equ:def-w}
\end{align}
where 
\begin{align}
g(s) &:=  \pi^{-s}\left(\frac{\Gamma(\frac{s}{2}+\frac{1}{4})}{\Gamma(\frac{1}{4})}\right)^2.
\label{equ:def-g}
\end{align}
Here, and henceforth, $\int_{(c)}$ stands for $\int_{c-i \infty}^{c+ i \infty}$. It can be shown (see \cite[Lemma 2.1]{Sound00}) that $w(\xi)$ is real-valued and smooth on $(0,+\infty)$, bounded as 
$\xi$ near 0, and decays exponentially as $\xi \rightarrow + \infty$. Define
\[
A (d):=\sum_{n=1}^{\infty}\frac{\tau(n)\chi_{8d}(n)}{\sqrt{n}}\omega \left(\frac{n\pi}{8d}\right),
\]
where $\tau(n)$ is the number of  divisors of $n$. It was proved \cite[Lemma 2.2]{Sound00} that for odd, positive, square-free integers $d$,
\begin{align}
L(\tfrac{1}{2},\chi_{8d})^2=2A(d).
\label{equ:2.1.1}
\end{align}

\subsection{Poisson summation formula} The following lemma is \cite[Lemma 2.2]{Sound-Young}.
\begin{lem} \label{lem:2.1}
 Let $\Phi$ be a smooth function with compact support on the positive real numbers, and suppose that $n$
 is an odd integer. Then
 \[
  \sum_{(d,2)=1}\left(\frac{d}{n}\right)\Phi\left(\frac{d}{Z}\right)=\frac{Z}{2n}\left(\frac{2}          {n}\right) \sum_{k\in \mathbb{Z}}(-1)^kG_{k}(n)\hat{\Phi}\left(\frac{kZ}{2n}\right),
 \]
where
 \begin{align}
  G_k(n):=\left(\frac{1-i}{2}+\left(\frac{-1}{n}\right)\frac{1+i}{2}\right)\sum_{a\, \mymod n}
  \left(\frac{a}{n}\right)e\left(\frac{ak}{n}\right),
  \label{equ:def-G}
 \end{align}
and
 \[
  \hat{\Phi}(y):=\int_{-\infty}^{\infty} \left(\cos(2\pi xy)+\sin(2\pi xy) \right)\Phi(x)dx
 \]
is a Fourier-type transform of $\Phi$.
\end{lem}
The precise values of the Gauss-type sum $G_k(n)$ have been calculated in \cite[Lemma 2.3]{Sound00} as follows.
\begin{lem}
 If $m$ and $n$ are relatively prime odd integers, then $G_k(mn)=G_k(m)G_k(n)$. Moreover, if $p^{\alpha}$ is 
 the largest power of $p$ dividing $k$ (setting $\alpha=\infty$ if $k=0$), then 
 \begin{align*}
 G_k(p^{\beta})=\left\{
 \begin{array}
  [c]{ll}
  0 & \text{if }\beta\leq\alpha \text{ is odd},\\
  \phi(p^{\beta}) & \text{if }\beta\leq\alpha \text{ is even},\\
  -p^{\alpha} & \text{if }\beta=\alpha+1 \text{ is even},\\
  \left(\frac{kp^{-\alpha}}{p}\right)p^{\alpha}\sqrt{p} & \text{if }\beta=\alpha+1 \text{ is odd},\\
  0 & \text{if }\beta\geq\alpha+2.
 \end{array}
 \right.
\end{align*}
\label{lem:2.2}
\end{lem}

\subsection{Smooth  function}\label{subsec:smooth}
Let $\Phi$ be a smooth Schwarz  class function that  is compactly  supported on $[\frac{1}{2}, \frac{5}{2}]$, and   $0 \leq  \Phi(t) \leq 1$ for all $t$. For any integer $\nu \geq 0$,  define 
\[
\Phi_{(\nu)} := \underset{0 \leq j \leq \nu}{\operatorname{ max }} \int_{\frac{1}{2}} ^{\frac{5}{2}} | \Phi^{(j)} (t)| dt.
\] 
For any $s \in \mathbb{C}$, define 
\[
\check{\Phi}(s) := \int_0^{\infty} \Phi(t) t^{-s} dt.
\]
Note that  $\check{\Phi}(s)$ is a holomorphic  function of $s$. Integrating by parts $\nu$ times gives us 
\[
\check{\Phi}(s)  = \frac{1}{(s-1) (s-2) \cdots (s-\nu)} \int_0^{\infty} \Phi^{(\nu)} (t) t^{-s+\nu} dt.
\] 
Hence, for $\Re(s)<1$, we see that 
\begin{equation}
\check{\Phi}(s) \ll _{\nu} \frac{3^{|\Re(s)|}}{|s-1|^\nu} \Phi_{(\nu)}.
\label{equ:bd for fei}
\end{equation}
\subsection{Some lemmas} 
The following lemma is the sharpest upper bound up to date for the fourth moment of quadratic Dirichlet $L$-functions, due to Heath-Brown \cite[Theorem 2]{HB01}.
\begin{lem}
Suppose $\sigma+it$ is a complex number with $\sigma \geq \frac{1}{2}$. Then
\[
\sideset{}{^\flat}\sum_{|d| \leq X}|L(\sigma+it,\chi_d)|^4 \ll  X^{1+\varepsilon}(1+|t|)^{1+\varepsilon}.
\]
\label{lem:2.3}
\end{lem}
Assuming GRH, the bound in Lemma \ref{lem:2.3} can  be improved by the following theorem.
\begin{thm}\label{thm:3.1}
Assume GRH for $L(s,\chi_d)$ for all fundamental discriminants $d$. Let $z_1,z_2 \in \mathbb{C}$ with $0 \leq \Re (z_1), \Re (z_2) \leq \frac{1}{\log X}$, and $|\Im (z_1)|, |\Im (z_2)| \leq X$. Then
 \[
  \sideset{}{^\flat}\sum_{|d|\leq X}  |L(\tfrac{1}{2}+z_1,\chi_d)|^2|L(\tfrac{1}{2}+z_2,\chi_d)|^2 
 \ll X(\operatorname{log}X)^{4+\varepsilon}\left(1+\min \left \{ (\log X)^6, \frac{1}{|\Im (z_1)- \Im (z_2)|^6} \right \} \right).
 \]
\end{thm}
Theorem  \ref{thm:3.1} is similar to \cite[Corollary 5.1]{Sound-Young}. Indeed, the proof of it follows closely the proof of \cite[Corollary 5.1]{Sound-Young}  and the argument in \cite[Section 4]{Sound01}.
Analogous  results to Theorem  \ref{thm:3.1} were obtained by Chandee \cite[Theorem 1.1]{Chandee} for the moments of the Riemann zeta function, and by Munsch  \cite[Theorem 1.1]{Munsch} for the moments of  Dirichlet $L$-functions modulo $q$. The proof of  Theorem  \ref{thm:3.1} is postponed to Section \ref{sec:upper bd}.

We remark that Lemma \ref{lem:2.3} is used to bound the error terms in the proof of Theorem \ref{main-thm1}, while both Lemma \ref{lem:2.3} and Theorem \ref{thm:3.1} are needed to bound the error terms in the proof of Theorem \ref{main-thm}.

\section{Setup of the problem}\label{sec:setup}
Let $\Phi$ be a smooth function as described in Subsection \ref{subsec:smooth}. We consider the following smoothed version of the fourth moment
\begin{align*}
\sideset{}{^*}\sum_{(d,2)=1}L(\tfrac{1}{2},\chi_{8d})^4 \Phi\left(\tfrac{d}{X}\right).
\end{align*}
Using the  approximate functional equation (\ref{equ:2.1.1}), we have
\begin{align}
\sideset{}{^*}\sum_{(d,2)=1}L(\tfrac{1}{2},\chi_{8d})^4 \Phi\left(\tfrac{d}{X}\right)
=\sideset{}{^*}\sum_{(d,2)=1}\left(A_{8d}(\tfrac{1}{2};8d)\right)^2\Phi\left(\tfrac{d}{X}\right),
\label{equ:13.0}
\end{align}
where
\begin{align}
A_{t}(\tfrac{1}{2};8d): = 2\sum_{n=1}^{\infty}\frac{\tau(n)\chi_{8d}(n)}{\sqrt{n}}\omega\left(\frac{n\pi}{t}\right).
\label{equ:13.1}
\end{align}
Let $X^{\frac{9}{10}} \leq U_1 \leq U_2 \leq X $ be  two parameters that will be chosen later.
Define
\begin{align}
S(U_1,U_2):=\sideset{}{^*}\sum_{(d,2)=1} A_{U_1}(\tfrac{1}{2};8d) A_{U_2}(\tfrac{1}{2};8d) \Phi\left(\tfrac{d}{X}\right).
\label{equ:13.1+}
\end{align}

 We remark that (\ref{equ:13.0}) is approximately equal to (\ref{equ:13.1+}) by choosing appropriate values for $U_1$ and $U_2$. This will be explained in Section \ref{sec:proof-main}.

Combining (\ref{equ:13.1}) and (\ref{equ:13.1+}), we  obtain that
\begin{align}
S(U_1,U_2)= 4\sideset{}{^*}\sum_{(d,2)=1}\sum_{n_1=1}^{\infty}\sum_{n_2=1}^{\infty}\frac{\tau(n_1)\tau(n_2)\chi_{8d}(n_1n_2)}{\sqrt{n_1n_2}}h(d,n_1,n_2),
\label{equ:13.2}
\end{align}
where 
\begin{align}
h(x,y,z) := \Phi\left(\frac{x}{X}\right)\omega\left(\frac{y\pi}{U_1}\right)\omega\left(\frac{z\pi}{U_2}\right).
\label{equ:def-h}
\end{align}
Using  the M\"{o}bius inversion to remove the square-free condition in (\ref{equ:13.2}) gives
\begin{align}
&S(U_1,U_2)\nonumber\\
&= 4\sum_{(d,2)=1} \sum_{a^2|d} \mu(a) \sum_{n_1=1}^{\infty}\sum_{n_2=1}^{\infty}\frac{\tau(n_1)\tau(n_2)\chi_{8d}(n_1n_2)}{\sqrt{n_1n_2}}h(d,n_1,n_2)\nonumber\\
&= 4\sum_{(a,2)=1} \mu(a) \sum_{(d,2)=1} \sum_{(n_1,a)=1}\sum_{(n_2,a)=1}\frac{\tau(n_1)\tau(n_2)\chi_{8d}(n_1n_2)}{\sqrt{n_1n_2}}h(a^2d,n_1,n_2)\nonumber\\
&=4\left( \sum_{\substack{a\leq Y \\ (a,2)=1}}+\sum_{\substack{a> Y \\ (a,2)=1}} \right) \mu(a) \sum_{(d,2)=1} \sum_{(n_1,a)=1}\sum_{(n_2,a)=1}\frac{\tau(n_1)\tau(n_2)\chi_{8d}(n_1n_2)}{\sqrt{n_1n_2}}h(a^2d,n_1,n_2)\nonumber\\
& =: S_1 +S_2.
\label{equ:13.3}
\end{align} 
In the above, we  let  $S_1$ denote the terms with $a \leq Y$,  where $Y$ is a parameter that satisfies $Y \leq X$. The value of $Y$  will be chosen later.  Also, we let  $S_2$ denote the terms with $a > Y $.  The terms $S_1$ contribute to the main term. We will  discuss $S_1$ in  Sections \ref{sec:k=0}, \ref{sec:k=square}, \ref{sec:error}. The terms $S_2$ contribute to the error  term by the  following lemma. 
\begin{lem}
Unconditionally, we have  $S_2 \ll X^{1+\varepsilon}Y^{-1}$. Under GRH, we have $S_2 \ll   XY^{-1} \log^{44}X$. 
\label{lem:13.1}
\end{lem}
\begin{proof}
Write $d=lb^2$, where $l$ is square-free and $b$ is positive. Grouping terms in $S_2$ according to $c=ab$, we deduce that 
\begin{align}
S_2&= 4\sum_{(c,2)=1}\sum_{\substack{a> Y \\ a|c}} \mu(a)\sideset{}{^*}\sum_{(l,2)=1}\sum_{(n_1,c)=1}\sum_{(n_2,c)=1}\frac{\tau(n_1)\tau(n_2)\chi_{8l}(n_1n_2)}{\sqrt{n_1n_2}}h (c^2l,n_1,n_2 ) \nonumber \\
&= \frac{4}{(2\pi i )^2}\sum_{(c,2)=1}\sum_{\substack{a> Y \\ a|c}} \mu(a)
\int_{(\frac{1}{2}+\varepsilon)} \int_{(\frac{1}{2}+\varepsilon)}   \frac{g(u)g(v)}{uv}  U_1^u U_2^v \nonumber\\
&\quad \times \sideset{}{^*}\sum_{(l,2)=1} \Phi\left( \frac{c^2l}{X}\right)
 L_c (\tfrac{1}{2}+u,\chi_{8l})^2  L_c (\tfrac{1}{2}+v,\chi_{8l})^2 \ du\ dv,
\label{equ:13.4}
\end{align}
where for $ \Re(s)>1$, $L_c(s,\chi)$ is given by the Euler product of $L(s,\chi)$ with omitting all prime factors of $c$. The last equation follows by the definition of $h(x,y,z)$ in (\ref{equ:def-h}).
Moving the lines of the integral to $\Re(u) = \Re(v) = \frac{1}{\log X}$, the double integral above is bounded by
\begin{align}
\ll (\log^2 X ) \tau^4(c) \int_{(\frac{1}{\log X})} \int_{(\frac{1}{\log X})}  |g(u)g(v)| \sideset{}{^*}\sum_{\substack{(l,2)=1 \\ l \leq \frac{5X}{2c^2}}}  |L (\tfrac{1}{2}+u,\chi_{8l} )|^4 \  |du| \ |dv|.
\label{equ:13.5}
\end{align}
Here we use the  inequalities $2ab \leq a^2 + b^2$ and 
$
|L_c(\tfrac{1}{2}+u,\chi_{8l})| \leq 
\tau(c) | L(\tfrac{1}{2}+u,\chi_{8l})|.
$

By Theorem \ref{thm:3.1}, we see that for $|\Im(u)| \leq \frac{X}{c^2}$,
\begin{align}
\sideset{}{^*}\sum_{\substack{(l,2)=1 \\ l \leq \frac{5X}{2c^2}}} |L(\tfrac{1}{2}+u,\chi_{8l}) |^4 \ll \frac{X}{c^2} \log^{11} X.
\label{equ:13.5+}
\end{align}
Also, by Lemma \ref{lem:2.3}, we get that 
\begin{align}
\sideset{}{^*}\sum_{\substack{(l,2)=1 \\ l \leq \frac{5X}{2c^2}}} |L(\tfrac{1}{2}+u,\chi_{8l})| ^4 
\ll \left( \frac{X}{c^2} \right)^{1+\varepsilon} (1+|\Im(u)|)^{1+\varepsilon}.
\label{equ:13.6}
\end{align}
Substituting both  (\ref{equ:13.5+})  and (\ref{equ:13.6})  in (\ref{equ:13.5}), we can bound (\ref{equ:13.5})  by
\begin{align*}
\ll \frac{\tau^4(c)}{c^2} X \log^{13} X.
\end{align*}
Together with (\ref{equ:13.4}),  this yields
\[
S_2 \ll X \log^{13} X \sum_{(c,2)=1} \frac{\tau^4(c)}{c^2}  \sum_{\substack{a>Y\\ a|c}} 1
\ll X \log^{13} X\sum_{c>Y}\frac{\tau^5(c)}{c^2}  \ll XY^{-1} \log^{44}X .
\]
This completes the proof of the conditional part of the lemma. The unconditional part follows similarly by   substituting (\ref{equ:13.6})  in (\ref{equ:13.5}).
\end{proof}
Now we consider $S_1$. Using the Poisson summation formula (see Lemma \ref{lem:2.1}) for the sum over $d$ in $S_1$, we obtain that
\begin{multline}
S_1=2X \sum_{\substack{a\leq Y \\ (a,2)=1}} \frac{\mu(a)}{a^2}\sum_{k \in \mathbb{Z}} (-1)^k \sum_{(n_1,2a)=1} \sum_{(n_2,2a)=1} \frac{\tau(n_1)\tau(n_2)}{\sqrt{n_1n_2}}\frac{G_k(n_1n_2)}{n_1n_2}\\
\times \int_{-\infty}^{\infty} h(xX,n_1,n_2) ( \cos + \sin ) \left( \frac{2\pi kxX}{2n_1n_2a^2}\right) dx. 
\label{equ:13.a}
\end{multline}

Let $S_1(k=0)$ denote the sum above over $k=0$, which are   called diagonal terms. Let $S_1(k\neq 0)$ denote the sum  over $k\neq 0$. Write $S_1(k\neq 0) = S_1(k= \Box) + S_1(k\neq \Box)$,  where $S_1(k= \Box)$ denotes the terms with square $k$, and $S_1(k \neq \Box)$ denotes the remaining terms. We call $S_1(k= \Box) $ off-diagonal terms. We will discuss $S_1(k=0)$,  $S_1(k= \Box) $, and  $S_1(k\neq \Box)$ in Section \ref{sec:k=0}, \ref{sec:k=square}, \ref{sec:error}, respectively.

\section{Evaluation of  $S_1(k=0)$} \label{sec:k=0}
In this section, we  shall   extract one main term of  $S_1$ from $S_1(k=0)$. The argument here is similar to \cite[Section 3.2]{Sound-Young}.

It follows  from  the  definition of $G_k(n)$ in (\ref{equ:def-G}) that $G_0(n) =  \phi(n)$ if $n=\Box$, and $G_0 (n) =  0$ otherwise.
By this fact and (\ref{equ:13.a}),  we see that 
\begin{align}
S_1(k=0)
&=2X \sum_{\substack{a\leq Y \\ (a,2)=1}} \frac{\mu(a)}{a^2}\sum_{\substack{(n_1n_2,2a)=1\\ n_1n_2 = \Box}} \frac{\tau(n_1)\tau(n_2)}{\sqrt{n_1n_2}}\frac{\phi(n_1n_2)}{n_1n_2}
\int_{-\infty}^{\infty} h(xX,n_1,n_2) dx \nonumber\\
&= 2X \sum_{\substack{(n_1n_2,2)=1\\ n_1n_2 = \Box}}\frac{\tau(n_1)\tau(n_2)}{\sqrt{n_1n_2}} \frac{\phi(n_1n_2)}{n_1n_2}\sum_{\substack{a\leq Y \\ (a,2n_1n_2)=1}} \frac{\mu(a)}{a^2}\int_{-\infty}^{\infty} h(xX,n_1,n_2) dx.
\label{equ:14.1}
\end{align}
Observe  that
\begin{align*}
\sum_{\substack{a\leq Y \\ (a,2n_1n_2)=1}} \frac{\mu(a)}{a^2} 
= \frac{8}{\pi^2}\prod_{p|n_1n_2}\left(1-\frac{1}{p^2} \right)^{-1} +O \left(Y^{-1}\right).
\end{align*}
Inserting this into (\ref{equ:14.1}), combined with 
\[
\frac{\phi(n_1n_2)}{n_1n_2} \prod_{p|n_1n_2} \left(1 - \frac{1}{p^2} \right)^{-1} = \prod_{p|n_1n_2} \frac{p}{p+1},
\]
we  obtain that
\begin{multline}
S_1(k=0)
=\frac{16X}{\pi^2}\sum_{\substack{(n_1n_2,2)=1\\ n_1n_2 = \Box}}\frac{\tau(n_1)\tau(n_2)}{\sqrt{n_1n_2}} \prod_{p|n_1n_2} \frac{p}{p+1}\int_{-\infty}^{\infty} h(xX,n_1,n_2) dx \\
 + O \left(\frac{X}{Y} \sum_{\substack{(n_1n_2,2)=1\\ n_1n_2 = \Box}}\frac{\tau(n_1)\tau(n_2)}{\sqrt{n_1n_2}} \int_{-\infty}^{\infty} |h(xX,n_1,n_2)| dx\right).
\label{equ:14.1+}
\end{multline}
Now we simplify the error term above. Recall that  $w(\xi)$ is  bounded as 
$\xi$ near 0 and decreases exponentially as $\xi \rightarrow + \infty$.  It follows that
\begin{align}
\sum_{\substack{(n_1n_2,2)=1\\ n_1n_2 = \Box}}\frac{\tau(n_1)\tau(n_2)}{\sqrt{n_1n_2}} \int_{-\infty}^{\infty} |h(xX,n_1,n_2)| dx 
&\ll
\sum_{\substack{(n_1n_2,2)=1\\ n_1n_2 = \Box}}\frac{\tau(n_1)\tau(n_2)}{\sqrt{n_1n_2}} \left( 1+\frac{n_1}{U_1} \right)^{-100} \left( 1+\frac{n_2}{U_2} \right)^{-100}  \nonumber\\
& \ll  \log^{11} X .
\label{equ:14.1++}
\end{align}
The last inequality follows by separating the  sum into two parts corresponding to whether $n_1, n_2 \leq U_1U_2$. Combining (\ref{equ:14.1+}) and  (\ref{equ:14.1++}),   we have
\begin{align*}
S_1(k=0)
=\frac{16X}{\pi^2}\sum_{\substack{(n_1n_2,2)=1\\ n_1n_2 = \Box}}\frac{\tau(n_1)\tau(n_2)}{\sqrt{n_1n_2}} \prod_{p|n_1n_2} \frac{p}{p+1}\int_{-\infty}^{\infty} h(xX,n_1,n_2)dx 
 +O \left(X Y^{-1}\log^{11} X \right).
\end{align*}
Recall $h(x,y,z)$ from (\ref{equ:def-h}) and  $\omega (\xi)$ from (\ref{equ:def-w}). We have
\begin{align}
&S_1(k=0) \nonumber \\
&=\frac{16X}{\pi^2}\int_{-\infty}^{\infty} \Phi\left( x \right) dx \frac{1}{(2\pi i)^2}\int_{(1)}\int_{(1)}\frac{g(u)g(v)}{uv}U_1^u U_2^v \sum_{\substack{(n_1n_2,2)=1\\ n_1n_2 = \Box}}\frac{\tau(n_1)\tau(n_2)}{n_1^{\frac{1}{2}+u}n_2^{\frac{1}{2}+v}} \prod_{p|n_1n_2} \frac{p}{p+1}\ du\ dv \nonumber \\
&\quad +O \left(X Y^{-1}\log^{11} X\right).
\label{equ:14.2}
\end{align}
\begin{lem}
For $\Re (\alpha),\Re (\beta) > \frac{1}{2}$, we have
\begin{align}
\sum_{\substack{(n_1n_2,2)=1\\ n_1n_2 = \Box}}\frac{\tau(n_1)\tau(n_2)}{n_1^{\alpha}n_2^{\beta}} \prod_{p|n_1n_2} \frac{p}{p+1}
= \zeta (2\alpha)^3 \zeta (2\beta)^3 \zeta (\alpha+\beta)^4  Z_1(\alpha,\beta),
\label{equ:Z}
\end{align}
where $Z_1(\alpha,\beta)$ is defined by
\begin{align*}
 Z_1(\alpha,\beta) := \prod_p Z_{1,p}(\alpha,\beta).
\end{align*}
Here
\begin{align*}
Z_{1,2}(\alpha,\beta) : =  \left( 1- \frac{1}{4^{\alpha}}\right)^3 \left( 1- \frac{1}{4^{\beta}}\right)^3 \left( 1- \frac{1}{2^{\alpha+\beta}}\right)^4,
\end{align*}
and for $p \nmid 2$,
\begin{multline*}
Z_{1,p}(\alpha,\beta) 
:= \left( 1- \frac{1}{p^{2\alpha}}\right)  \left( 1- \frac{1}{p^{2\beta}}\right)  \left( 1- \frac{1}{p^{\alpha+\beta}}\right)^4
\Bigg[ 
1 + \frac{4}{p^{\alpha+\beta}} + \frac{1}{p^{2\alpha}} + \frac{1}{p^{2\beta}} + \frac{1}{p^{2\alpha+2\beta}}
  - \frac{1}{p+1} \\
\times\left( 
\frac{3}{p^{2\alpha}} 
+ \frac{3}{p^{2\beta}} + \frac{4}{p^{\alpha+\beta}}  -  \frac{1}{p^{4\alpha}} -  \frac{1}{p^{4\beta}} - \frac{3}{p^{2\alpha+2\beta}} + \frac{2}{p^{2\alpha+4\beta}} + \frac{2}{p^{4\alpha+2\beta}}  - \frac{1}{p^{4\alpha+4\beta}}
\right)
\Bigg] .
\end{multline*}

Furthermore, $Z_1(\alpha,\beta)$ is analytic and uniformly bounded in the region $\Re (\alpha),\Re (\beta) \geq \frac{1}{4}+\varepsilon$.
\label{lem:Z_1}
\end{lem}
\begin{proof}
We have
\begin{equation*}
\begin{split}
\sum_{\substack{(n_1n_2,2)=1\\ n_1n_2 = \Box}}\frac{\tau(n_1)\tau(n_2)}{n_1^{\alpha}n_2^{\beta}} \prod_{p|n_1n_2} \frac{p}{p+1}= \prod_{(p,2)=1}\left( 1 + \frac{p}{p+1}\left( \sum_{r=1}^{\infty}\sum_{n_1 n_2=p^{2r}}\frac{\tau(n_1)\tau(n_2)}{n_1^{\alpha}n_2^{\beta}}\right)\right).
\end{split}
\end{equation*}
Note that
\begin{align*}
\sum_{r=1}^{\infty}\sum_{n_1 n_2=p^{2r}}\frac{\tau(n_1)\tau(n_2)}{n_1^{\alpha}n_2^{\beta}}
= \frac{(1+\frac{1}{p^{2\alpha}})(1+\frac{1}{p^{2\beta}})}{(1-\frac{1}{p^{2\alpha}})^2(1-\frac{1}{p^{2\beta}})^2}  + \frac{4}{p^{\alpha+\beta}} \frac{1}{(1-\frac{1}{p^{2\alpha}})^2(1-\frac{1}{p^{2\beta}})^2} -1.
\end{align*}
Thus,
\begin{multline*}
\sum_{\substack{(n_1n_2,2)=1\\ n_1n_2 = \Box}}\frac{\tau(n_1)\tau(n_2)}{n_1^{\alpha}n_2^{\beta}} \prod_{p|n_1n_2} \frac{p}{p+1}
= \prod_{(p,2)=1} \frac{1}{(1-\frac{1}{p^{2\alpha}})^2(1-\frac{1}{p^{2\beta}})^2}
\Bigg[ 
1 + \frac{4}{p^{\alpha+\beta}} + \frac{1}{p^{2\alpha}} + \frac{1}{p^{2\beta}} + \frac{1}{p^{2\alpha+2\beta}}\\
- \frac{1}{p+1} 
\left( 
\frac{3}{p^{2\alpha}} 
+ \frac{3}{p^{2\beta}} + \frac{4}{p^{\alpha+\beta}} 
-  \frac{1}{p^{4\alpha}}
 -  \frac{1}{p^{4\beta}}  - \frac{3}{p^{2\alpha+2\beta}} + \frac{2}{p^{2\alpha+4\beta}} + \frac{2}{p^{4\alpha+2\beta}}  - \frac{1}{p^{4\alpha+4\beta}}
\right)
\Bigg].
\end{multline*}
Then (\ref{equ:Z}) follows  by comparing Euler factors on  both  sides. The remaining part of the lemma   follows directly from  the  definition of $Z_1(\alpha,\beta)$.
\end{proof}
It follows from  (\ref{equ:14.2}) and    Lemma \ref{lem:Z_1} that
\begin{multline}
\label{equ:14.2+} 
S_1(k=0) 
=\frac{16X}{\pi^2}\int_{-\infty}^{\infty} \Phi\left( x \right) dx \frac{1}{(2\pi i)^2}\int_{(1)}\int_{(1)}\frac{g(u)g(v)}{uv}U_1^u U_2^v \zeta (1+2u)^3 \zeta (1+2v)^3 \zeta (1+u+v) ^4\\
\times Z_1(\tfrac{1}{2}+u,\tfrac{1}{2}+v) \ du\ dv 
+O \left(XY^{-1}\log^{11} X\right).
\end{multline}
The double integral in (\ref{equ:14.2+}) can be written as  
\begin{align*}
\frac{1}{(2\pi i)^2}\int_{(1)}\int_{(1)}  \frac{U_1^u U_2^v}{uv(2u)^3(2v)^3(u+v)^4} \mathcal{E}(u,v)  \ du\ dv ,
\end{align*}
where 
\[
\mathcal{E}(u,v) := g(u)g(v)  \zeta(1+2u)^3 (2u)^3   \zeta(1+2v)^3 (2v)^3  \zeta(1+u+v)^4 (u+v)^4  Z_1(\tfrac{1}{2}+u,\tfrac{1}{2}+v).
\]
Clearly, $\mathcal{E}$ is analytic for $\Re(u),\Re(v)\geq  -\frac{1}{4} + \varepsilon$.

Now move the lines of the integral above to $\Re(u) = \Re(v) = \frac{1}{10}$ without encountering any poles. Next move the line of the integral over $v$ to $\Re(v) = -\frac{1}{5}$. We  may encounter two poles of   order at most $4$ at both $v=0$ and $v=-u$. Thus,
\begin{align}
&\frac{1}{(2\pi i)^2}\int_{(\frac{1}{10})}\int_{(\frac{1}{10})}  \frac{U_1^u U_2^v}{uv(2u)^3(2v)^3(u+v)^4} \mathcal{E}(u,v) \ du\ dv  \nonumber\\
&= \frac{1}{2\pi i}\int_{(\frac{1}{10})} \left( \underset{v=0}{\operatorname{Res}} +\underset{v=-u}{\operatorname{Res}}\right)
 \left[\frac{U_1^u U_2^v}{uv(2u)^3(2v)^3(u+v)^4}\mathcal{E}(u,v)  \right]du + O \left(U_1^{\frac{1}{10}} U_2^{-\frac{1}{5}} \right).
 \label{equ:14.3}
\end{align}

The integral of the residue at $v=-u$ in (\ref{equ:14.3}) will  contribute to an error term. In fact, we have 
\begin{align*}
&\underset{v=-u}{\operatorname{Res}}
 \left[\frac{U_1^u U_2^v}{uv(2u)^3(2v)^3(u+v)^4}\mathcal{E}(u,v)  \right]\\
& = \frac{1}{3!}\left. \frac{\partial^3}{\partial v^3} \right|_{v=-u} \left[\frac{U_1^u U_2^v}{uv(2u)^3(2v)^3}\mathcal{E}(u,v)  \right] \\
 &= \frac{U_1^u U_2^{-u} }{384 u^{11} } 
  \Big [
    \mathcal{E}(u, -u) \left( u^3 \log^3 U_2 + 12 u^2 \log^2 U_2 + 60 u \log U_2 + 120\right) \\
   &\quad + 
    \mathcal{E}^{(0,1)}(u, -u) \left(  3u^3 \log^2 U_2  + 24 u^2 \ln U_2 + 60 u\right)  \\
   &\quad +
   \mathcal{E}^{(0,2)}(u, -u) \left( 3u^3 \log U_2 + 12 u^2\right) 
   +
   \mathcal{E}^{(0,3)}(u, -u)u^3 
 \Big ],
\end{align*}
where $ \mathcal{E}^{(i,j)}(u, v) := \frac{\partial^{i+j}\mathcal{E}}{\partial u^i \partial v^j} (u,v)$.
It follows that
\begin{align}
\frac{1}{2\pi i}\int_{(\frac{1}{10})} \underset{v=-u}{\operatorname{Res}}
 \left[\frac{U_1^u U_2^v}{uv(2u)^3(2v)^3(u+v)^4}\mathcal{E}(u,v)  \right]du \ll U_1^{\frac{1}{10}} U_2^{-\frac{1}{10}} \log^3 X. 
 \label{equ:k=0 residue-error}
\end{align}

It remains to  consider the integral of the residue at $v=0$ in (\ref{equ:14.3}).   Note that
\begin{align*}
I_1(u)&:=\underset{v=0}{\operatorname{Res}} \left[\frac{U_1^u U_2^v}{uv(2u)^3(2v)^3(u+v)^4}\mathcal{E}(u,v)  \right]\\
&= \frac{U_1^u}{384u^{11}} 
\Big[ 
\mathcal{E}(u,0) (u^3 \log^3 U_2 -12 u^2 \log^2 U_2 + 60 u \log U_2 -120 ) \\
& \quad +
\mathcal{E}^{(0,1)}(u,0)(3u^3 \log^2 U_2 -24 u^2  \log U_2 + 60 u )\\
&\quad + 
\mathcal{E}^{(0,2)}(u,0)(3u^3 \log U_2 - 12u^2) + \mathcal{E}^{(0,3)}(u,0) u^3
\Big] .
\end{align*}
Moving the line of the  integral below
 from $\Re(u)= \frac{1}{10}$ to $\Re(u)= -\frac{1}{10}$  with encountering a pole at $u=0$, we see that
\begin{align}
\frac{1}{2\pi i}\int_{(\frac{1}{10})} I_1(u) du 
&=
 \underset{u=0}{\operatorname{Res}} \ I_1(u)
  +O(U_1^{-\frac{1}{10}}\log^3 X) \nonumber\\
&= \frac{\mathcal{E}(0,0)}{11612160} \left(-\log^{10} U_1 + 5 \log^9 U_1 \log U_2 - 9\log^8 U_1 \log^2 U_2
+ 6 \log^7 U_1 \log^3 U_2
 \right) \nonumber\\
 &\quad +O\left(\log^9 X \right)+O \left(U_1^{-\frac{1}{10}}\log^3 X \right).
\label{equ:14.4}
\end{align}

Combining (\ref{equ:14.2+}), (\ref{equ:14.3}), (\ref{equ:k=0 residue-error}) and (\ref{equ:14.4}), we obtain that
\begin{multline}
S_1(k=0)\\
=\frac{16X}{\pi^2}\tilde{\Phi}(1) \cdot  \frac{\mathcal{E}(0,0)}{11612160} \left(-\log^{10} U_1 + 5 \log^9 U_1 \log U_2 - 9\log^8 U_1 \log^2 U_2
+ 6 \log^7 U_1 \log^3 U_2
 \right) \\ 
 + O\left( X \log^9 X + XY^{-1} \log^{11} X \right).
 \label{equ:revise-1}
\end{multline}
where $\tilde{\Phi}(s)$ is  defined in (\ref{def tilde f}).

Now we compute $\mathcal{E}(0,0) $ above. Clearly, $\mathcal{E}(0,0) = Z_1(\frac{1}{2},\frac{1}{2}) $. By the definition of $Z_1(u,v)$ in Lemma \ref{lem:Z_1}, it follows that
\begin{align}
Z_1(\tfrac{1}{2},\tfrac{1}{2}) &= \frac{1}{2^{10}} \prod_{(p,2)=1} \left(1- \frac{1}{p} \right)^6 \left[ 1 + \frac{6}{p} + \frac{1}{p^2} - \frac{1}{p+1} \left( \frac{10}{p} - \frac{5}{p^2} + \frac{4}{p^3}  - \frac{1}{p^4}\right)\right] \nonumber\\
&= \frac{1}{2^{10}} \prod_{(p,2)=1} \frac{(1- \frac{1}{p} )^6 }{1+ \frac{1}{p}} \left( 1 + \frac{7}{p} - \frac{3}{p^2} + \frac{6}{p^3} - \frac{4}{p^4}  + \frac{1}{p^5}\right).
\label{equ:Z-1-value}
\end{align}
On the other hand, recalling the definition of $a_4$ from (\ref{equ:a-k}), we have
\begin{align}
a_4 &=  \frac{1}{2^{12}} 
\prod_{(p,2)=1} \frac{(1-\frac{1}{p})^{10}}{1+\frac{1}{p}} 
\left( 
\frac{
 (1+\frac{1}{\sqrt{p}} )^{-4} +
 (1-\frac{1}{\sqrt{p}} )^{-4}
 }{2}
 +\frac{1}{p}
\right)\nonumber\\
&=\frac{1}{2^{12}} 
\prod_{(p,2)=1} \frac{(1-\frac{1}{p})^{10}}{1+\frac{1}{p}}  \frac{1}{(1+\frac{1}{\sqrt{p}})^4 (1-\frac{1}{\sqrt{p}})^4}
\left[
\frac{1}{2} \left (1-\frac{1}{\sqrt{p}} \right)^4 + \frac{1}{2 } \left(1+\frac{1}{\sqrt{p}} \right)^4 + \frac{1}{p} \left( 1 -  \frac{1}{p}\right)^4 
\right]\nonumber\\
&= \frac{1}{2^{12}} 
\prod_{(p,2)=1} \frac{(1-\frac{1}{p})^{6}}{1+\frac{1}{p}}  
\left(
1 + \frac{7}{p} - \frac{3}{p^2} + \frac{6}{p^3} - \frac{4}{p^4}  + \frac{1}{p^5}
\right).
\label{equ:a-4-value}
\end{align}
Comparing (\ref{equ:Z-1-value}) with (\ref{equ:a-4-value}), we  conclude $ Z_1(\frac{1}{2},\frac{1}{2}) = 4 a_4$, which implies 
$\mathcal{E}(0,0) = 4 a_4$.
Together with  (\ref{equ:revise-1}), it follows that 
\begin{lem} \label{lem:k=0}
We have
\begin{multline*}
S_1(k=0)
= \frac{a_{4} \tilde{\Phi}(1)  X}{2^6 \cdot 3^4 \cdot 5 \cdot 7 \cdot \pi^2} \left(-\log^{10} U_1 + 5 \log^9 U_1 \log U_2 - 9\log^8 U_1 \log^2 U_2
+ 6 \log^7 U_1 \log^3 U_2
 \right) \\
 +O\left( X \log^9 X + XY^{-1} \log^{11} X \right).
\end{multline*}
\end{lem}

\section{Evaluation of  $S_1(k=\Box)$} \label{sec:k=square}
In this section, we compute another part of the main term of $S_1$ which arises from $S_1(k=\square)$. Many of the techniques 
used here are from  Sections 5.2,   5.3 of \cite{Sound00}.

Recall  from (\ref{equ:13.a}) that
\begin{multline}
S_1(k \neq 0)=2X \sum_{\substack{a\leq Y \\ (a,2)=1}} \frac{\mu(a)}{a^2}\sum_{ k\neq 0} (-1)^k \sum_{(n_1,2a)=1} \sum_{(n_2,2a)=1} \frac{\tau(n_1)\tau(n_2)}{\sqrt{n_1n_2}}\frac{G_k(n_1n_2)}{n_1n_2}\\
\times \int_{-\infty}^{\infty} h(xX,n_1,n_2)(\cos  + \sin )\left( \frac{2\pi kxX}{2n_1n_2a^2}\right) dx.
\label{equ:15.1}
\end{multline}

To proceed, we need the following lemma.
\begin{lem}
Let $f(x)$ be a smooth function on $\mathbb{R}_{> 0}$. Suppose $f$ decays rapidly as $x \rightarrow \infty$, and $f^{(n)}(x)$ converges as $x \rightarrow 0^+$ for every $n\in \mathbb{Z}_{\geq 0}$.
Then we have 
\begin{equation}
\int_0^{\infty} f(x)\cos (2\pi xy) dx=\frac{1}{2\pi i} \int_{(\frac{1}{2})} \tilde{f}(1-s) \Gamma(s) \cos \left(\frac{\operatorname{sgn}(y)\pi s}{2}\right) (2\pi|y|)^{-s}ds,
\label{equ:cos-sin-2}
\end{equation}
where $\tilde{f}$ is the Mellin transform of $f$ defined  by
\begin{equation}
\tilde{f}(s) := \int_0^{\infty} f(x) x^{s-1} dx.
\label{def tilde f}
\end{equation}
In addition, the equation \eqref{equ:cos-sin-2} is also valid when  $\cos$ is replaced by $\sin$.
\label{lem:15.1}
\end{lem}
\begin{proof}
See   \cite[Section 3.3]{Sound-Young}.
\end{proof}
Taking $f(x) = h(xX, n_1, n_2)$ in Lemma \ref{lem:15.1}, we have 
\begin{align*}
& \int_{-\infty}^{\infty} h(xX,n_1,n_2) (\cos+\sin) \left( \frac{2\pi kxX}{2n_1n_2a^2}\right)dx 
 \\
& = \frac{X^{-1}}{2\pi i}\int_{(\frac{1}{2})}\tilde{h}(1-s;n_1,n_2)\Gamma(s)
 (\cos+\operatorname{sgn}(k)\sin)\left(\frac{\pi s}{2}\right)\left(\frac{n_1n_2a^2}{\pi |k|}\right)^{s}ds,
\end{align*}
where 
\[
\tilde{h}(1-s;n_1,n_2) =  \int_0^{\infty} h(x, n_1 , n_2)x^{-s} dx.
\]
Recall from (\ref{equ:def-h}) the definition of $h(x,y,z)$. The above contour integral is 
\begin{equation*}
\frac{1 }{(2\pi i)^3 }\int_{(\varepsilon)}\int_{(\varepsilon)}\int_{(\varepsilon)}
 \Gamma(s)\left(\frac{a^2}{ |k|}\right)^s \mathcal{J}(s,k) g(u)g(v) \frac{1}{n_1^{u-s} n_2^{v-s}}\frac{U_1^u U_2^vX^{-s}}{uv}\ ds \ du \ dv, 
\end{equation*}
where
\[
\mathcal{J}(s,k) =   \tilde{\Phi}(1-s) (\cos+\operatorname{sgn}(k)\sin)\left(\frac{\pi s}{2}\right) \pi^{-s}  .
\]
Move the lines of the triple integral to $\Re(s) = \frac{1}{2}+\varepsilon$, $\Re(u)=\Re(v)=\frac{1}{2}+2\varepsilon$, and change the variables $u'=u-s$, $v'=v-s$. We obtain that
\begin{align*}
&\int_{-\infty}^{\infty} h(xX,n_1,n_2) (\cos+\sin) \left( \frac{2\pi kxX}{2n_1n_2a^2}\right)dx
\\
&=\frac{1 }{(2\pi i)^3 }\int_{(\varepsilon)}\int_{(\varepsilon)}\int_{(\frac{1}{2}+\varepsilon)}
 \Gamma(s)\left(\frac{a^2}{ |k|}\right)^s \mathcal{J}(s,k) g(u+s)g(v+s) \frac{1}{n_1^{u} n_2^{v}}\frac{U_1^{u+s} U_2^{v+s} X^{-s}}{(u+s)(v+s)}\ ds \ du \ dv.
\end{align*}
Substituting this in (\ref{equ:15.1}), we get  that
\begin{multline}
S_1(k \neq 0)=2X \sum_{\substack{a\leq Y \\ (a,2)=1}} \frac{\mu(a)}{a^2}\sum_{ k\neq 0} (-1)^k 
\frac{1 }{(2\pi i)^3 }\int_{(\varepsilon)}\int_{(\varepsilon)}\int_{(\frac{1}{2}+\varepsilon)}
\Gamma(s)\left(\frac{a^2}{ |k|}\right)^s \mathcal{J}(s,k) g(u+s)g(v+s)\\
\times \frac{U_1^{u+s} U_2^{v+s} X^{-s}}{(u+s)(v+s)}
 \sum_{(n_1,2a)=1} \sum_{(n_2,2a)=1} \frac{\tau(n_1)\tau(n_2)}{n_1^{\frac{1}{2}+u}n_2^{\frac{1}{2}+v}}\frac{G_k(n_1n_2)}{n_1n_2}
 \ ds \ du \ dv.
  \label{equ:15.2}
\end{multline}
\begin{lem}
Write $4k = k_1k_2^2$, where $k_1$ is a fundamental discriminant
(possibly $k_1 = 1$), and $k_2$ is a positive integer.
In the region $\Re(\alpha),\Re(\beta)>\frac{1}{2}$, we have
\begin{align}
\sum_{(n_1,2a)=1}\sum_{(n_2,2a)=1} \frac{\tau(n_1)\tau(n_2)}{n_1^{\alpha}n_2^{\beta}}\frac{G_k(n_1n_2)}{n_1n_2} = L(\tfrac{1}{2}+\alpha,\chi_{k_1})^2  L(\tfrac{1}{2}+\beta,\chi_{k_1})^2  Z_2(\alpha,\beta,a,k).
\label{equ:Z_2}
\end{align}
Here $Z_2(\alpha,\beta,a,k)$ is defined as follows:
\begin{align*}
Z_2(\alpha,\beta,a,k):=\prod_{p}Z_{2,p}(\alpha,\beta,a,k),
\end{align*}
where 
\begin{align*}
Z_{2,p}(\alpha,\beta,a,k) := \left( 1-\frac{\chi_{k_1}(p)}{p^{\frac{1}{2}+\alpha}}\right)^2\left( 1-\frac{\chi_{k_1}(p)}{p^{\frac{1}{2}+\beta}}\right)^2  \qquad \text{if } p|2a,
\end{align*}
and 
\begin{align*}
Z_{2,p}(\alpha,\beta,a,k) := \left( 1-\frac{\chi_{k_1}(p)}{p^{\frac{1}{2}+\alpha}}\right)^2\left( 1-\frac{\chi_{k_1}(p)}{p^{\frac{1}{2}+\beta}}\right)^2 \sum_{n_1=0}^{\infty} \sum_{n_2=0}^{\infty} \frac{\tau(p^{n_1})\tau(p^{n_2})}{p^{n_1\alpha + n_2\beta}}\frac{G_k(p^{n_1+n_2})}{p^{n_1+n_2}} \qquad \text{if } p\nmid 2a.
\end{align*}

In addition, $Z_2(\alpha,\beta,a,k)$ is analytic in the region $\Re(\alpha),\Re(\beta)>0$, and we have
\begin{equation}
Z_2(\alpha,\beta,a,k) \ll \tau^4(a) \tau^{8}(|k|) \log^{10} X
\label{equ:15.111111}
\end{equation}
in the region $\Re(\alpha),\Re(\beta) \geq \frac{1}{\log X}$,
where the implied constant is absolute.
\label{lem:15.2}
\end{lem}
\begin{proof}
The formula (\ref{equ:Z_2}) follows from the joint multiplicativity of $G_k(n_1n_2)$ with variables $n_1$ and $n_2$. In fact,
\begin{align*}
\sum_{(n_1,2a)=1}\sum_{(n_2,2a)=1} \frac{\tau(n_1)\tau(n_2)}{n_1^{\alpha}n_2^{\beta}}\frac{G_k(n_1n_2)}{n_1n_2} 
&=  \prod_{(p,2a)=1}\sum_{n_1=0}^{\infty} \sum_{n_2=0}^{\infty} \frac{\tau(p^{n_1})\tau(p^{n_2})}{p^{n_1\alpha + n_2\beta}}\frac{G_k(p^{n_1+n_2})}{p^{n_1+n_2}} .
\end{align*}
Then we obtain (\ref{equ:Z_2}) by comparing Euler factors on both sides.

 For $p \nmid 2a k$, by Lemma \ref{lem:2.2}, we know that
\begin{align}
Z_{2,p}(\alpha,\beta,a,k) =  \left( 1-\frac{\chi_{k_1}(p)}{p^{\frac{1}{2}+\alpha}}\right)^2\left( 1-\frac{\chi_{k_1}(p)}{p^{\frac{1}{2}+\beta}}\right)^2  \left( 1+\frac{2\chi_{k_1}(p)}{p^{\frac{1}{2}+\alpha}}+\frac{2\chi_{k_1}(p)}{p^{\frac{1}{2}+\beta}} \right).
\label{equ:15.2+}
\end{align}
This shows that   $Z_2(\alpha,\beta,a,k)$ is analytic in the region $\Re(\alpha),\Re(\beta)>0$.

It remains to prove the upper bound of $Z_2(\alpha,\beta,a,k)$. It follows from (\ref{equ:15.2+}) that for 
$p \nmid 2a k$,
\[
\prod_{(p,2ak)=1} Z_{2,p}(\alpha,\beta,a,k) = \prod_{(p,2ak)=1} \left( 1- \frac{3}{p^{1+2\alpha}} - \frac{3}{p^{1+2\beta}} - \frac{4}{p^{1+\alpha+\beta}} + O \left( \frac{1}{p^{3/2}}\right) \right) \ll \log^{10} X.
\]
For $p|2a$, we get  that
\[
\prod_{p|2a} Z_{2,p}(\alpha,\beta,a,k)  \leq  \prod_{p|2a}  \left( 1+\frac{1}{\sqrt{p}}\right)^4 \ll \tau^4(a).
\]
For $p \nmid 2a,p|k$, using the trivial bound $G_k(p^n) \leq  p^n$, we obtain  that
\begin{align*}
\prod_{ p|k, p \nmid 2a} Z_{2,p}(\alpha,\beta,a,k) 
\leq  \prod_{ p|k, p \nmid 2a} \left( 1+ \frac{1}{\sqrt{p}}\right)^4 \sum_{0 \leq n_1 + n_2 \leq \operatorname{ord}_p (k) + 1 } 
(n_1 +1)(n_2+1) \ll \tau^{8}(|k|). 
\end{align*}
By the above three bounds, we have obtained  (\ref{equ:15.111111}).
\end{proof}
By (\ref{equ:15.2}) and Lemma \ref{lem:15.2}, it follows that
\begin{align}
&S_1(k \neq 0)\nonumber\\
&=\frac{2X }{(2\pi i)^3 } \sum_{\substack{a\leq Y \\ (a,2)=1}} \frac{\mu(a)}{a^2}\sum_{ k\neq 0} (-1)^k \int_{(\varepsilon)}\int_{(\varepsilon)}\int_{(\frac{1}{2}+\varepsilon)} \Gamma(s) \mathcal{J}(s,k) a^{2s} g(u+s)g(v+s) \frac{U_1^{u+s} U_2^{v+s}X^{-s}}{(u+s)(v+s)}\nonumber \\
& \quad\times \frac{1}{|k|^s}  L(1+u,\chi_{k_1})^2  L(1+v,\chi_{k_1})^2 Z_2(\tfrac{1}{2}+u,\tfrac{1}{2}+v,a,k)\ ds \ du \ dv.
\label{equ:15.3}
\end{align}
Note that when moving the lines of integration of the variables $u,v$ to the left, then we may encounter poles only when $k=\Box$ (then $k_1=1$). Thus, we break the sum in $(\ref{equ:15.3})$ into two parts depending on whether $k=\Box$. 

Write 
\begin{multline*}
S_1(k=\Box) 
 :=\frac{2X }{(2\pi i)^3 } \sum_{\substack{a\leq Y \\ (a,2)=1}} \frac{\mu(a)}{a^2}\sum_{\substack{ k \neq 0 \\ k=\Box} } (-1)^k \int_{(\varepsilon)}\int_{(\varepsilon)}\int_{(\frac{1}{2}+\varepsilon)} \Gamma(s) \mathcal{J}(s,k) a^{2s} g(u+s)g(v+s) \\
 \times \frac{U_1^{u+s} U_2^{v+s}X^{-s}}{(u+s)(v+s)}
 \frac{1}{|k|^s}  \zeta(1+u)^2  \zeta(1+v)^2 Z_2(\tfrac{1}{2}+u,\tfrac{1}{2}+v,a,k)\ ds \ du \ dv,
\end{multline*}
and 
\begin{multline}
S_1(k \neq \Box) 
  :=\frac{2X }{(2\pi i)^3 } \sum_{\substack{a\leq Y \\ (a,2)=1}} \frac{\mu(a)}{a^2}\sum_{ k\neq 0, \Box} (-1)^k \int_{(\varepsilon)}\int_{(\varepsilon)}\int_{(\tfrac{1}{2}+\varepsilon)} \Gamma(s) \mathcal{J}(s,k) a^{2s} g(u+s)g(v+s)\\
 \times  \frac{U_1^{u+s} U_2^{v+s}X^{-s}}{(u+s)(v+s)} 
 \frac{1}{|k|^s}  L(1+u,\chi_{k_1})^2  L(1+v,\chi_{k_1})^2 Z_2(\tfrac{1}{2}+u,\tfrac{1}{2}+v,a,k)\ ds \ du \ dv.
\label{equ:15.4}
\end{multline}

We will give an upper bound for $S_1(k \neq \Box)$ in the next section. In the rest of this section, we focus on  $S_1(k=\Box)$ and obtain a main term.
By the change of variables (replace $k$ by $k^2$), we get that
\begin{multline}
S_1(k=\Box) 
 =\frac{2X }{(2\pi i)^3 } \sum_{\substack{a\leq Y \\ (a,2)=1}} \frac{\mu(a)}{a^2}\sum_{k=1}^{
 \infty} (-1)^k \int_{(\varepsilon)}\int_{(\varepsilon)}\int_{(\tfrac{1}{2}+\varepsilon)} \Gamma(s) \mathcal{J}(s,1) a^{2s} g(u+s)g(v+s) \\
\times \frac{U_1^{u+s} U_2^{v+s}X^{-s}}{(u+s)(v+s)} 
 \frac{1}{k^{2s}}  \zeta(1+u)^2 \zeta(1+v)^2  Z_2(\tfrac{1}{2}+u,\tfrac{1}{2}+v,a,k^2)\ ds \ du \ dv.
\label{equ:15.5}
\end{multline}
\begin{lem} 
In the region $\Re(\alpha),\Re(\beta)>0$,  $\Re(\gamma)>\tfrac{1}{2}$,
\begin{align}
\sum_{k=1}^{\infty} \frac{(-1)^k}{k^{2\gamma}}Z_2(\alpha,\beta,a,k^2) = 
(2^{1-2 \gamma}-1)\zeta(2 \gamma)   Z_3(\alpha,\beta,\gamma,a).
\label{equ:Z_3}
\end{align}
Here $Z_3(\alpha,\beta,\gamma,a)$ is defined by
\[
Z_3(\alpha,\beta,\gamma,a) := \zeta(2 \alpha +2 \gamma)^2  \zeta(2 \beta +2 \gamma)^2 \prod_p Z_{3,p}(\alpha,\beta,\gamma,a),
\]
where for $p|2a$,
\begin{align}
Z_{3,p}(\alpha,\beta,\gamma,a): = \left( 1- \frac{1}{p^{\frac{1}{2}+\alpha}}\right)^2  \left( 1- \frac{1}{p^{\frac{1}{2}+\beta}}\right)^2  \left( 1- \frac{1}{p^{2\alpha + 2\gamma}}\right)^2  \left( 1- \frac{1}{p^{2\beta + 2\gamma}}\right)^2,
\label{equ:15.5+}
\end{align}
and for $p \nmid 2a$,
\begin{multline}
Z_{3,p}(\alpha,\beta,\gamma,a)
:= \left( 1- \frac{1}{p^{\frac{1}{2}+\alpha}}\right)^2  \left( 1- \frac{1}{p^{\frac{1}{2}+\beta}}\right)^2  \Bigg[ \left(1-\frac{1}{p}\right)\left(1+\frac{1}{p^{2\alpha+2\gamma}}\right)\left(1+\frac{1}{p^{2\beta+2\gamma}}\right) \\
+ \frac{1}{p}\left(1-\frac{1}{p^{2\alpha+2\gamma}}\right)^2\left(1-\frac{1}{p^{2\beta+2\gamma}}\right)^2 
+ \left(1-\frac{1}{p}\right)\frac{4}{p^{\alpha+\beta+2\gamma}}\\
+ 
2\left( 1-\frac{1}{p^{2\gamma}}\right)\left( \frac{1}{p^{\frac{1}{2}+\alpha}} + \frac{1}{p^{\frac{1}{2}+\beta}} + \frac{1}{p^{\frac{1}{2}+2\alpha+\beta + 2\gamma}}+ \frac{1}{p^{\frac{1}{2}+\alpha+2\beta + 2\gamma}}\right)\Bigg].
\label{equ:15.5++}
\end{multline}

Moreover,  
\begin{itemize}
\item[(1)] $Z_3(\alpha,\beta,\gamma,a)$  is analytic and uniformly bounded in the region $\Re(\alpha),\Re(\beta)\geq \tfrac{1}{2}+\varepsilon$, $\Re(\gamma)\geq 2\varepsilon$.

\item[(2)] $Z_3(\alpha,\beta,\gamma,a)$  is analytic  and $Z_3(\alpha,\beta,\gamma,a) \ll  \log^{14} X $ in the region $\Re(\alpha),\Re(\beta)\geq \tfrac{1}{2}+\frac{1}{\log X}$, $\Re(\gamma)\geq \frac{2}{\log X}$. The implied constant is absolute.
\end{itemize}
\label{lem:15.3}
\end{lem}
\begin{proof}
We first compute the left-hand side of  (\ref{equ:Z_3}) without $(-1)^k$.  Note that
\begin{align}
\sum_{k=1}^{\infty} \frac{1}{k^{2\gamma}}Z_2(\alpha,\beta,a,k^2) &= \sum_{k=1}^{\infty}  \frac{1}{k^{2\gamma}}\prod_pZ_{2,p}(\alpha,\beta,a,k^2) = \prod_p \sum_{b=0}^{\infty} \frac{Z_{2,p}(\alpha,\beta,a,p^{2b})}{p^{2b\gamma}}.
\label{equ:without-1}
\end{align}
We remark here that $Z_{2,p} (\alpha,\beta,a,1)$ may not be $1$.
If $p|2a$,   we have
\begin{align}
\sum_{b=0}^{\infty} \frac{Z_{2,p}(\alpha,\beta,a,p^{2b})}{p^{2b\gamma}} 
= \frac{1}{1-\frac{1}{p^{2\gamma}}} \left( 1- \frac{1}{p^{\frac{1}{2}+\alpha}}\right)^2\left( 1- \frac{1}{p^{\frac{1}{2}+\beta}}\right)^2.
\label{equ:Z-2-1}
\end{align}
If $p \nmid 2a$, we have 
\begin{multline}
\sum_{b=0}^{\infty} \frac{Z_{2,p}(\alpha,\beta,a,p^{2b})}{p^{2b\gamma}}
= \left( 1- \frac{1}{p^{\frac{1}{2}+\alpha}}\right)^2\left( 1- \frac{1}{p^{\frac{1}{2}+\beta}}\right)^2 \\ 
\quad \times\sum_{b=0}^{\infty} \frac{1}{p^{2b\gamma}} 
\left(
\sum_{\substack{n_1,n_2 \geq 0 \\ n_1+n_2 = 2b+1}}  \frac{\tau(p^{n_1}) \tau(p^{n_2})}{p^{n_1\alpha+n_2\beta}} \frac{p^{2b}\sqrt{p}}{p^{n_1+n_2}} 
+  \sum_{\substack{n_1,n_2 \geq 0 \\ n_1+n_2 \leq 2b \\ n_1+n_2 \text{ even}}}  \frac{\tau(p^{n_1}) \tau(p^{n_2})}{p^{n_1\alpha+n_2\beta}} \frac{\phi(p^{n_1+n_2})}{p^{n_1+n_2}}
\right).
\label{equ:Z-2-1-1}
\end{multline}
Note that
\begin{align*}
&\sum_{b=0}^{\infty} \frac{1}{p^{2b\gamma}} \sum_{\substack{n_1,n_2 \geq 0 \\ n_1+n_2 = 2b+1}}  \frac{\tau(p^{n_1}) \tau(p^{n_2})}{p^{n_1\alpha+n_2\beta}} \frac{p^{2b}\sqrt{p}}{p^{n_1+n_2}} \\
&=\frac{1}{p^{-\gamma+\frac{1}{2}}}
\frac{1}{(1-\frac{1}{p^{2\alpha+2\gamma}})^2(1-\frac{1}{p^{2\beta+2\gamma}})^2} 
\left[
\frac{2}{p^{\alpha+\gamma}}\left( 1+ \frac{1}{p^{2\beta +2\gamma}}\right)
+ 
\frac{2}{p^{\beta+\gamma}}\left( 1+ \frac{1}{p^{2\alpha +2\gamma}}\right)
\right],
\end{align*}
and 
\begin{multline*}
\sum_{b=0}^{\infty} \frac{1}{p^{2b\gamma}} 
 \sum_{\substack{n_1,n_2 \geq 0 \\ n_1+n_2 \leq 2b \\ n_1+n_2 \text{ even}}}  \frac{\tau(p^{n_1}) \tau(p^{n_2})}{p^{n_1\alpha+n_2\beta}} \frac{\phi(p^{n_1+n_2})}{p^{n_1+n_2}} 
 = 
 \frac{1}{1-\frac{1}{p^{2\gamma}}} \\
\quad  \times  \left[
\frac{1}{p}+
\left(1-\frac{1}{p} \right)
\frac{1}{(1-\frac{1}{p^{2\alpha+2\gamma}})^2 (1-\frac{1}{p^{2\beta+2\gamma}})^2}  
\left(
 \left( 1+\frac{1}{p^{2\alpha+2\gamma}} \right) \left( 1+\frac{1}{p^{2\beta+2\gamma}} \right) +\frac{4}{p^{\alpha+\beta+2\gamma}}
 \right)
 \right].
\end{multline*}
Inserting them into (\ref{equ:Z-2-1-1}), combined  with (\ref{equ:without-1}), (\ref{equ:Z-2-1}),  we obtain that 
\begin{align}
\sum_{k=1}^{\infty} \frac{1}{k^{2\gamma}}Z_2(\alpha,\beta,a,k^2)
= \zeta(2 \gamma)   Z_3(\alpha,\beta,\gamma,a).
\label{equ:without -1}
\end{align}

Now we prove   (\ref{equ:Z_3}).  It is clear that $G_{4k}(n) = G_k(n)$ for any odd $n$, so $Z_2(\alpha,\beta,a,4k^2)=Z_2(\alpha,\beta,a,k^2)$. Thus,
\begin{align*}
\sum_{k=1}^{\infty} \frac{(-1)^k}{k^{2\gamma}}Z_2(\alpha,\beta,a,k^2) &= \frac{1}{4^\gamma }\sum_{k=1}^{\infty} \frac{1}{k^{2\gamma}}Z_2(\alpha,\beta,a,4k^2) - \sum_{k \text{ odd}} \frac{1}{k^{2\gamma}}Z_2(\alpha,\beta,a,k^2)\\
&= (2^{1-2 \gamma}-1)\sum_{k=1}^{\infty} \frac{1}{k^{2\gamma}}Z_2(\alpha,\beta,a,k^2).
\end{align*}
Together with (\ref{equ:without -1}), this yields  (\ref{equ:Z_3}).

The first property of $Z_3(\alpha,\beta,\gamma,a)$ comes directly from its definition. Now we prove  the second property. We know  that  for $\Re(\alpha),\Re(\beta)\geq \tfrac{1}{2}+\frac{1}{\log X}$,  $\Re(\gamma)\geq \frac{2}{\log X}$,
\[
Z_3(\alpha,\beta,\gamma,a) \ll  (\log^4 X) \prod_{p|2a} \left( 1+ \frac{1}{p^{1+\frac{1}{\log X}}} \right)^8
\prod_{p \nmid 2a} \left( 1+ \frac{6}{p^{1+\frac{6}{\log X}}} +  \frac{4}{p^{1+\frac{5}{\log X}}} + O\left( \frac{1}{p^2} \right)\right) \ll \log^{14} X,
\] 
as desired.
\end{proof}
It follows from (\ref{equ:15.5})    and Lemma \ref{lem:15.3}  that
\begin{multline*}
S_1(k=\Box) 
  =\frac{2X }{(2\pi i)^3 } \sum_{\substack{a\leq Y \\ (a,2)=1}} \frac{\mu(a)}{a^2}\int_{(\varepsilon)}\int_{(\varepsilon)}\int_{(\frac{1}{2}+\varepsilon)} \Gamma(s) \mathcal{J}(s,1) a^{2s} g(u+s)g(v+s) \frac{U_1^{u+s} U_2^{v+s} X^{-s}}{(u+s)(v+s)}   \\
  \times \zeta(1+u)^2 \zeta(1+v)^2  (2^{1-2s}-1)\zeta(2s)   Z_3(\tfrac{1}{2}+u,\tfrac{1}{2}+v, s,a)\ ds \ du \ dv.
\end{multline*}

Note that $Z_3(\tfrac{1}{2}+u,\tfrac{1}{2}+v, s)$ is analytic in the region $\Re(u),\Re(v) \geq \varepsilon$, $\Re(s) \geq 2\varepsilon$ by  (1) of Lemma \ref{lem:15.3}, so we move the lines of the integral above to $\Re(u) = \Re(v) = 1$, $\Re(s) =\frac{1}{10}$ without encountering any poles. (The only possible pole lies in $\zeta(2s)$ at $s=\frac{1}{2}$, but is cancelled by the simple zero arising from $2^{1-2s} -1 $.) Hence,
\begin{multline}
S_1(k=\Box) 
  =\frac{2X }{(2\pi i)^3 } \sum_{\substack{a\leq Y \\ (a,2)=1}} \frac{\mu(a)}{a^2}\int_{(1)}\int_{(1)}\int_{(\frac{1}{10})} \Gamma(s) \mathcal{J}(s,1) a^{2s} g(u+s)g(v+s) \frac{U_1^{u+s} U_2^{v+s} X^{-s}}{(u+s)(v+s)}   \\
 \times \zeta(1+u)^2 \zeta(1+v)^2  (2^{1-2s}-1)\zeta(2s)   Z_3(\tfrac{1}{2}+u,\tfrac{1}{2}+v, s,a)\ ds \ du \ dv.
\label{equ:15.9-}
\end{multline}

Note that  we may  extend the  sum  over $a$ to infinity with an error  term
\begin{align*}
&\frac{2X }{(2\pi i)^3 } \sum_{\substack{a > Y \\ (a,2)=1}} \frac{\mu(a)}{a^2}\int_{(1)}\int_{(1)}\int_{(\frac{1}{10})} \Gamma(s) \mathcal{J}(s,1) a^{2s} g(u+s)g(v+s) \frac{U_1^{u+s} U_2^{v+s}X^{-s}}{(u+s)(v+s)}  \nonumber \\
&\times  \zeta(1+u)^2 \zeta(1+v)^2  (2^{1-2s}-1)\zeta(2s)   Z_3(\tfrac{1}{2}+u,\tfrac{1}{2}+v, s,a)\ ds \ du \ dv.
\end{align*}
Move the lines of the integral above to $\Re(u)=\Re(v)=\frac{1}{\log X}$,  $\Re(s) = \frac{2}{\log X}$ without encountering any poles. Then by (2) of Lemma \ref{lem:15.3},  this is bounded by 
\begin{align*}
&\ll X \log^{20}X \sum_{\substack{a > Y \\ (a,2)=1}} \frac{1}{a^{2-\frac{4}{\log X}}} \\
& \quad \times \int_{(\frac{1}{\log X})}\int_{(\frac{1}{\log X})}\int_{(\frac{2}{\log X})}\
 (1+|2s|) |\mathcal{J}(s,1)|  \left|\Gamma(s)\Gamma(\tfrac{u+s}{2}+\tfrac{1}{4})^2  \Gamma(\tfrac{v+s}{2}+\tfrac{1}{4})^2 \right|\ |ds| \ |du| \ |dv| \\
&\ll X (\log^{20} X) Y^{-1} \int_{(\frac{2}{\log X})} (1+|2s|)  |\Gamma(s)| |\tilde{\Phi}(1-s)| \left|(\cos + \sin )(\tfrac{\pi s}{2})\right| |ds|\\
&\ll X  Y^{-1} (\log^{21} X) \Phi_{(5)}.
\end{align*}
The last inequality is due to (\ref{equ:bd for fei}) and the fact $| \Gamma(s) (\cos+\sin)(\frac{\pi s }{2}) |\ll |s|^{\Re(s)-\frac{1}{2}}$.
Together with (\ref{equ:15.9-}), it implies that
\begin{multline}
S_1(k=\Box)
  =\frac{2X }{(2\pi i)^3 } \sum_{(a,2)=1} \frac{\mu(a)}{a^2}\int_{(1)}\int_{(1)}\int_{(\frac{1}{10})} \Gamma(s) \mathcal{J}(s,1) a^{2s} g(u+s)g(v+s) \frac{U_1^{u+s} U_2^{v+s}X^{-s}}{(u+s)(v+s)}  \\
\quad \times  \zeta(1+u)^2 \zeta(1+v)^2  (2^{1-2s}-1)\zeta(2s)   Z_3(\tfrac{1}{2}+u,\tfrac{1}{2}+v, s,a)\ ds \ du \ dv \\
+ O \left(  X  Y^{-1} (\log^{21} X) \Phi_{(5)} \right) . 
\label{equ:15.9}
\end{multline}

Let $K_1(\alpha,\beta,\gamma;p), K_2(\alpha,\beta,\gamma;p)$ denote the expressions of (\ref{equ:15.5+}) and (\ref{equ:15.5++}), respectively. We have the following lemma.
\begin{lem} \label{lem:Z-4}
In the region $\Re(\alpha),\Re(\beta)>\tfrac{1}{2}$, $0 < \Re(\gamma) < \tfrac{1}{2} ,$
\begin{align}
\sum_{(a,2)=1} \frac{\mu(a)}{a^{2-2\gamma}} Z_3(\alpha,\beta, \gamma, a) = \frac{ \zeta(2\alpha+2\gamma)^3 \zeta(2\beta+2\gamma)^3 \zeta(\alpha+\beta+2\gamma)^4 }{\zeta(\frac{1}{2}+\alpha+2\gamma)^2 \zeta(\frac{1}{2}+\beta+2\gamma)^2 } Z_4(\alpha,\beta,\gamma),
\label{equ:Z-4}
\end{align}
where
\begin{align*}
&Z_4(\alpha,\beta,\gamma) 
:= K_1(\alpha,\beta,\gamma;2) \\
 &\quad \times \prod_p \frac{(1-\frac{1}{p^{2\alpha+2\gamma}})(1-\frac{1}{p^{2\beta+2\gamma}})(1-\frac{1}{p^{\alpha+\beta+2\gamma}})^4}
{ (1-\frac{1}{p^{\frac{1}{2}+\alpha+2\gamma}})^{2} (1-\frac{1}{p^{\frac{1}{2}+\beta+2\gamma}})^{2} } 
\prod_{(p,2)=1} \left( K_2(\alpha,\beta,\gamma;p)-\frac{1}{p^{2-2\gamma}}K_1(\alpha,\beta,\gamma;p) \right).
\end{align*}

Moreover, $Z_4(\alpha,\beta,\gamma)$ is analytic and uniformly bounded in the region $\Re(\alpha),\Re(\beta)\geq \frac{3}{8}$, $-\frac{1}{16}  \leq \Re(\gamma) \leq \frac{1}{8}$.
\label{lem:15.4}
\end{lem}
\begin{proof}
We have
\begin{equation*}
\begin{split}
&\sum_{(a,2)=1} \frac{\mu(a)}{a^{2-2\gamma}} Z_3 (\alpha,\beta,\gamma,a) \\
&= \zeta(2\alpha+2\gamma)^2 \zeta(2\beta+2\gamma)^2 \sum_{(a,2)=1} \frac{\mu(a)}{a^{2-2\gamma}} \prod_{p|2a} K_1(\alpha,\beta,\gamma;p) \prod_{p \nmid 2a} K_2(\alpha,\beta,\gamma;p) \\
&= \zeta(2\alpha+2\gamma)^2 \zeta(2\beta+2\gamma)^2 K_1(\alpha,\beta,\gamma;2)\prod_{(p,2)=1} \left(K_2(\alpha,\beta,\gamma;p) - \frac{1}{p^{2-2\gamma}} K_1(\alpha,\beta,\gamma;p)\right).
\end{split}
\end{equation*}
This implies the equation (\ref{equ:Z-4}). The later part of the lemma   can  be proved directly  by the definition of  $Z_4(\alpha,\beta,\gamma)$.
\end{proof}
It follows from (\ref{equ:15.9}) and  Lemma \ref{lem:15.4} that
\begin{align}
S_1(k=\Box)
 & =\frac{2X }{(2\pi i)^3 }\int_{(1)}\int_{(1)}\int_{(\frac{1}{10})}  \mathcal{J}(s,1) (2^{1-2s}-1)\zeta(2s) g(u+s)g(v+s) \frac{U_1^{u+s} U_2^{v+s}X^{-s}}{(u+s)(v+s)}  \nonumber \\
&\quad \times  \zeta^2(1+u)\zeta^2(1+v)  
 \frac{\Gamma(s) \zeta(1+2u+2s)^3 \zeta(1+2v+2s)^3  \zeta(1+u+v+2s)^4 }{\zeta(1+u+2s)^2 \zeta(1+v+2s)^2 } \nonumber\\ 
&\quad \times Z_4(\tfrac{1}{2}+u,\tfrac{1}{2}+v,s) \ ds \ du \ dv + O \left(X  Y^{-1}(\log^{21} X ) \Phi_{(5)} \right),
\label{equ:15.10}
\end{align}
where 
$Z_4(\tfrac{1}{2}+u,\tfrac{1}{2}+v,s)$ is analytic and uniformly bounded in the region  $\Re(u),\Re(v)\geq -\frac{1}{8}$,  $-\frac{1}{16}  \leq \Re(s) \leq \frac{1}{8} $. 

Move the lines of the triple integral above  to $\Re(u) = \Re(v) = \Re(s) = \frac{1}{100}$ without encountering any poles. Then move the line of the integral over $v$ to $\Re(v)= -\frac{1}{50}  + \frac{1}{\log X}$.  There is a pole of order at most  $2$ at $v=0$, and a pole of  order at most $4$ at $v=-s$, so the triple integral  in (\ref{equ:15.10}) is 
\begin{align}
  \frac{1}{(2\pi i)^2} \int_{(\frac{1}{100})} \int_{(\frac{1}{100})} I_2  (u,s) + I_3 (u,s) 
 \ du \ ds + O \left(U_1^{\frac{1}{50}} U_2^{-\frac{1}{100}}X^{-\frac{1}{100}} (\log^2 X)  \Phi_{(5)} \right),
\label{equ:15.12}
\end{align} 
where
$I_2(u,s),I_3(u,s)$ are the residues of the integrand in (\ref{equ:15.10}) at $v=0$ and $v=-s$, respectively. 

The double integral of $I_3 (u,s)$ in (\ref{equ:15.12}) is bounded. To see this, note that
\begin{multline*}
I_3(u,v)=   \mathcal{J}(s,1) (2^{1-2s}-1)\zeta(2s) g(u+s)  \frac{U_1^{u+s}X^{-s}}{u+s}  
 \frac{\zeta(1+u)^2  \Gamma(s) \zeta(1+2u+2s)^3}{\zeta(1+u+2s)^2} \frac{1}{3!}\left.  \frac{d^3}{dv^3} \right|_{v=-s} \\
 \left(\frac{g(v+s) U_2^{v+s}\zeta(1+v)^2 \zeta(1+2v+2s)^3(v+s)^3 \zeta(1+u+v+2s)^4 Z_4(\tfrac{1}{2}+u,\tfrac{1}{2}+v,s)}{\zeta(1+v+2s)^2 } \right).
\end{multline*}
Moving the line of the following  integral in terms  of $u$ from  $\Re(u)=\frac{1}{100}$ to $\Re(u)=\frac{1}{\log X}$ gives
\begin{align}
 \frac{1}{(2\pi i)^2} \int_{(\frac{1}{100})} \int_{(\frac{1}{100})}  I_3 (u,s) 
 \ du \ ds
\ll  U_1^{\frac{1}{100}}  X^{-\frac{1}{100}} ( \log^5   X) \Phi_{(5)}.
 \label{equ:15.13}
\end{align}

Now we  handle the double integral of $I_2 (u,s)$ in (\ref{equ:15.12}). Write the integrand in  (\ref{equ:15.10}) in the form  of   
\[
\frac{U_1^{u+s} U_2^{v+s}X^{-s}}{(u+s)(v+s)}    \frac{1}{u^2v^2}   \frac{(u+2s)^2 (v+2s)^2}{s (2u+2s)^3 (2v+2s)^3 (u+v+2s)^4}\mathcal{F}(u,v,s).
\]
Clearly, $\mathcal{F}(u,v,s)$ is analytic in the region  $\Re(u+2s),\Re(v+2s)>0$,   $\Re(u),\Re(v)\geq -\frac{1}{8}$ and  $-\frac{1}{16}  \leq \Re(s) \leq \frac{1}{8} $.
We have
\begin{multline*}
I_2(u,s)
= \frac{U_1^{u+s} U_2^{s} X^{-s}}{ 16(u+2s)^3 (u+s)^4 s^4 u^2}\\
\times \left[
\mathcal{F} (u,0,s) ( us \log U_2 + 2s^2 \log U_2 -10s -3u)
+
\mathcal{F}^{(0,1,0)} (u,0,s)  (us+2s^2)
\right].
\end{multline*}
Move the line of the double integral below from $\Re(u)= \frac{1}{100} $ to $\Re(u)= -\frac{1}{100} + \frac{1}{\log X}$. There is  one possible pole at $u=0$. Hence, 
\begin{align}
\frac{1}{(2\pi i)^2} \int_{(\frac{1}{100})} \int_{(\frac{1}{100})} 
I_2(u,s)  \ du\ ds
= \frac{1}{2\pi i} \int_{(\frac{1}{100})}
\underset{u=0}{\operatorname{Res}} 
 \left( I_2(u,s) \right) ds
+ O \left(U_2^{\frac{1}{100}}X^{-\frac{1}{100}} \log^{8} X \right).
\label{equ:15.14}
\end{align}
Note that
\begin{multline*}
\underset{u=0}{\operatorname{Res}}\ I_2(u,s)
= \frac{U_1^{s} U_2^s X^{-s}}{64 s^{11} }
 \Big(
\mathcal{F}(0,0,s)(s^2\log U_1 \log U_2  
-5  s\log U_1 
-5 s \log U_2
+26
  )\\
+\mathcal{F}^{(1,0,0)}(0,0,s)(s^2 \log U_2  
-5   s)
+\mathcal{F}^{(0,1,0)}(0,0,s)(s^2\log U_1  
-5   s)
+ \mathcal{F}^{(1,1,0)}(0,0,s) s^2
\Big).
\end{multline*}
We see that the expression in the brackets above  is analytic for $-\frac{1}{16}  \leq \Re(s) \leq \frac{1}{8} $. Then 
we move the line of the integral below to $\Re(s)=-\frac{1}{100}$ with only a possible pole at $s=0$, and get that
\begin{multline}
\frac{1}{2\pi i} \int_{(\frac{1}{100})}
\underset{u=0}{\operatorname{Res}} 
 \left( I_2(u,s) \right) ds
=
\frac{\mathcal{F}(0,0,0)}{64} \sum_{\substack{j_1 + j_2 +j_3 +j_4 = 10 \\ j_1,j_2,j_3,j_4 \geq 0}} \frac{(-1)^{j_3} B(j_4)}{j_1 ! j_2 ! j_3 ! j_4 !} 
 (\log^{j_1} U_1)  (\log^{j_2} U_2) (\log^{j_3} X)\\
 + O \left( U_1^{-\frac{1}{100}} U_2^{-\frac{1}{100}} X^{\frac{1}{100}} \log^2 X+ \log^9 X\right),
 \label{equ:15.15}
\end{multline}
where 
 \begin{equation}
  \begin{split}
 B(j)=\left\{
 \begin{array}
  [c]{ll}
     26                    & \text{if } j=0,\\
  -5 (\log U_1 + \log U_2) & \text{if } j=1,\\
   2\log U_1 \log U_2      & \text{if } j= 2,\\
  0                        & \text{if } j\geq 3.
 \end{array}
 \right.
 \end{split}
 \label{equ:B(j)}
\end{equation}

Next we compute $\mathcal{F}(0,0,0)$ above. Note that $\mathcal{F}(0,0,0)  = \mathcal{J}(0,1) g(0)^2  Z_4 (\frac{1}{2}, \frac{1}{2}, 0) = -\frac{1}{2}\tilde{\Phi}(1) Z_4 (\frac{1}{2}, \frac{1}{2}, 0)$. Recalling the definition of $Z_4 (\alpha, \beta, \gamma)$ from Lemma \ref{lem:Z-4}, we have
\begin{align*}
&Z_4 (\tfrac{1}{2}, \tfrac{1}{2}, 0) \\
&= K_1 ( \tfrac{1}{2}, \tfrac{1}{2}, 0; 2 ) \prod_{p} \left( 1 - \frac{1}{p}\right)^2 \prod_{(p,2)=1} \left( K_2 ( \tfrac{1}{2}, \tfrac{1}{2},0 ; p)- \frac{1}{p^2}K_1 (\tfrac{1}{2}, \tfrac{1}{2}, 0;p) \right) \\
&=  \frac{1}{4} K_1 (\tfrac{1}{2}, \tfrac{1}{2}, 0;2 ) \prod_{(p,2)=1}   \left( 1 - \frac{1}{p}\right)^7  \left( 1 + \frac{7}{p} - \frac{3}{p^2} + \frac{6}{p^3} - \frac{4}{p^4} + \frac{1}{p^5}\right)\\
&= \frac{1}{4 \zeta_2 (2)} K_1 ( \tfrac{1}{2}, \tfrac{1}{2}, 0;2 ) \prod_{(p,2)=1}  \frac{ ( 1 - \frac{1}{p})^6}{1+ \frac{1}{p}}  \left( 1 + \frac{7}{p} - \frac{3}{p^2} + \frac{6}{p^3} - \frac{4}{p^4} + \frac{1}{p^5}\right)\\
&=\frac{32a_4 }{\pi^2}.
\end{align*}
The last equality is due to (\ref{equ:a-4-value}). Thus, 
\[
\mathcal{F}(0,0,0) = - \frac{16 \tilde{\Phi}(1) a_4}{\pi^2}. 
\]

Combining (\ref{equ:15.10}) with (\ref{equ:15.12}), (\ref{equ:15.13}), (\ref{equ:15.14}), (\ref{equ:15.15}), and the identity above,  it follows that
\begin{lem} \label{lem:k=square}
We have
\begin{multline*}
S_1(k=\Box) 
  = - \frac{a_4 \tilde{\Phi}(1) X}{2 \pi^2}    \sum_{j_1 + j_2 +j_3 +j_4 = 10} \frac{(-1)^{j_3} B(j_4)}{j_1 ! j_2 ! j_3 ! j_4 !} 
 (\log^{j_1} U_1)  (\log^{j_2} U_2) (\log^{j_3} X) \\
 + X 
 \cdot O \left(  
   \log^{9} X 
 + U_1^{\frac{1}{100}}  X^{-\frac{1}{100}} ( \log^5   X) \Phi_{(5)}
 + U_1^{\frac{1}{50}} U_2^{-\frac{1}{100}} X^{-\frac{1}{100}} (\log^2 X) \Phi_{(5)}
 +  Y^{-1} (\log ^{21} X)  \Phi_{(5)}
   \right) .
\end{multline*}
\end{lem}

\section{Upper bounds for  $S_1(k \neq \Box)$}\label{sec:error}
In this section, we shall prove  the following upper bounds for $S_1(k \neq \square)$. The techniques applied here are from \cite[Section 5.4]{Sound00} and the last part of \cite[Section 3]{Sound-Young}.
\begin{lem}\label{lem:error}
Unconditionally, we have
\[
S_1(k \neq \Box) \ll U_1^{\frac{1}{2}} U_2^{\frac{1}{2}}Y X^{\varepsilon}  \Phi_{(5)}.
\]
Under GRH, we have 
\[
S_1(k \neq \Box) \ll U_1^{\frac{1}{2}} U_2^{\frac{1}{2}} Y (\log X)^{2^{17}} \Phi_{(5)}.
\]
\end{lem}
\begin{proof}
It follows from (\ref{equ:15.4}) that 
\begin{multline}
S_1(k \neq \Box) 
\ll  X\sum_{\substack{a\leq Y \\ (a,2)=1}} \frac{1}{a^2}\sideset{}{^\flat}\sum_{k_1 \neq 0,1}\sum_{k_2=1}^{\infty} \Bigg | \int_{(\varepsilon)}\int_{(\varepsilon)}\int_{(\frac{1}{2}+\varepsilon)} \Gamma(s) \mathcal{J}(s,k_1) a^{2s} g(u+s)g(v+s) \frac{U_1^{u+s} U_2^{v+s} X^{-s}}{(u+s)(v+s)} \\
\times \frac{4^s}{|k_1k_2^2|^s}  L(1+u,\chi_{k_1})^2  L(1+v,\chi_{k_1})^2 Z_2(\tfrac{1}{2}+u,\tfrac{1}{2}+v,a,k_1k_2^2)\ ds \ du \ dv \Bigg |.
\label{equ:16.1}
\end{multline}
Separate the sum over $k_1$ to the sum over  $|k_1| \leq T:= U_1 U_2 Y^2 X^{-1}$, and that over  $|k_1| > T$. Clearly, $X^{\frac{4}{5}} \leq T \leq X^3$ since $X^{\frac{9}{10}} \leq U_1 \leq  U_2 \leq X$ and $ 1\leq Y \leq X$. For the first category, we move the the lines of the integral to 
$\Re(u)=\Re(v) = -\frac{1}{2}+ \frac{1}{4\log X}$, $\Re(s)= \frac{3}{4}$, while for the second category, we move the lines to $\Re(u)=\Re(v) = -\frac{1}{2}+ \frac{1}{4\log X}$,  $\Re(s)= \frac{5}{4}$. 

By (\ref{equ:15.111111}), the terms in the first category are bounded by 
\begin{multline}
\ll  X^{\frac{1}{4}} U_1^{\frac{1}{4}} U_2^{\frac{1}{4}} \log^{10} X\sum_{\substack{a\leq Y \\ (a,2)=1}} \frac{\tau^{4}(a)}{\sqrt{a}}\int_{(-\frac{1}{2}+\frac{1}{\log X})}\int_{(-\frac{1}{2}+\frac{1}{\log X})}\int_{(\frac{3}{4})} |\mathcal{J}(s,k_1) \Gamma(s)g(u+s)g(v+s)| \\
\times \sideset{}{^\flat}\sum_{ |k_1| \leq  T}\frac{\tau^{8}(|k_1|)}{|k_1|^{\frac{3}{4}}}  |L(1+u,\chi_{k_1}) |^4 \ |ds| \ |du| \ |dv|.
\label{equ:16.2}
\end{multline}

Note that 
\begin{align}
\sideset{}{^\flat} \sum_{ |k_1| \leq  T}\frac{\tau^{8}(|k_1|)}{|k_1|^{\frac{3}{4}}}  |L(1+u,\chi_{k_1}) |^4 
 \ll \sum_{1\leq 2^l \leq T} \sideset{}{^\flat} \sum_{2^l < |k_1| \leq 2^{l+1}} \frac{\tau^{8}(|k_1|)}{|k_1|^{\frac{3}{4}}}|L(1+u,\chi_{k_1})|^4.
\label{equ:16.3a-}
\end{align}
By (\ref{equ:16.3a-})  and Lemma  \ref{lem:2.3},  it follows that
\begin{align}
\sideset{}{^\flat} \sum_{ |k_1| \leq  T}\frac{\tau^{8}(|k_1|)}{|k_1|^{\frac{3}{4}}}  |L(1+u,\chi_{k_1}) |^4 
 \ll T^{\frac{1}{4}+\varepsilon} 
(1+|\Im(u)|)^{1+\varepsilon}.
\label{equ:16.3a}
\end{align}

This bound can be improved under GRH.
In fact, 
we split the left-hand side of (\ref{equ:16.3a}) into
\begin{align*}
\sideset{}{^\flat} \sum_{ |k_1| \leq  T}\frac{\tau^{8}(|k_1|)}{|k_1|^{\frac{3}{4}}}  |L(1+u,\chi_{k_1}) |^4 
=\sideset{}{^\flat} \sum_{ |k_1| \leq  X^{\frac{1}{5}}}\frac{\tau^{8}(|k_1|)}{|k_1|^{\frac{3}{4}}}  |L(1+u,\chi_{k_1}) |^4  + \sideset{}{^\flat} \sum_{ X^{\frac{1}{5}}<|k_1| \leq  T}\frac{\tau^{8}(|k_1|)}{|k_1|^{\frac{3}{4}}}  |L(1+u,\chi_{k_1}) |^4.
\end{align*}
By Theorem \ref{thm:3.1}, we have for $|\Im(u)| \leq X^{\frac{1}{5}}$,
\[
\sideset{}{^\flat} \sum_{ |k_1| \leq  X^{\frac{1}{5}}}\frac{\tau^{8}(|k_1|)}{|k_1|^{\frac{3}{4}}}  |L(1+u,\chi_{k_1}) |^4 \ll X^{\frac{1}{5}} \log^{11} X.
\]
Later in (\ref{equ:trivial-upper-L}) of Section \ref{sec:upper bd}, under GRH, it will be proved that for $-\frac{1}{2}\leq \Re(u) \leq -\frac{1}{2} + \frac{1}{\log X}$ and $|\Im(u)|\leq X$,
\begin{align*}
\sideset{}{^\flat}\sum_{|k_1| \leq X}|L(1+u,\chi_{k_1}) |^8  \ll X  (\log X)^{37}.
\end{align*}
Using dyadic blocks and Cauchy-Schwarz inequality,  combined with the above bound, we can deduce that for $|\Im(u)| \leq  X^{\frac{1}{5}}$,
\[
\sideset{}{^\flat} \sum_{ X^{\frac{1}{5}}<|k_1| \leq  T}\frac{\tau^{8}(|k_1|)}{|k_1|^{\frac{3}{4}}}  |L(1+u,\chi_{k_1}) |^4 \ll T^{\frac{1}{4}} \log^{2^{16}} X.
\]
Thus for $|\Im(u)| \leq  X^{\frac{1}{5}}$,
\begin{align}
\sideset{}{^\flat} \sum_{ |k_1| \leq T}\frac{\tau^{8}(|k_1|)}{|k_1|^{\frac{3}{4}}}  |L(1+u,\chi_{k_1}) |^4 \ll T^{\frac{1}{4}} \log^{2^{16}} X.
\label{equ:16.3b}
\end{align}
Recall the definition of $T$. Substituting both (\ref{equ:16.3a}) and (\ref{equ:16.3b}) in (\ref{equ:16.2}), we have proved the contribution of the terms in the first category is  $ \ll U_1^{\frac{1}{2}} U_2^{\frac{1}{2}} Y (\log X)^{2^{17}} \Phi_{(5)}$.  Similarly, we can  deduce  that the  contribution of the terms in the second category is also $ \ll U_1^{\frac{1}{2}} U_2^{\frac{1}{2}} Y (\log X)^{2^{17}} \Phi_{(5)}$.

The conditional part of the lemma is  proved now. The unconditional part can be proved  similarly by  substituting (\ref{equ:16.3a}) in (\ref{equ:16.2}).
\end{proof}

\section{Proof of  main theorems}  \label{sec:proof-main}
In this section, we complete the proof of Theorem \ref{main-thm} and Theorem \ref{main-thm1}. The argument  is similar to \cite[Section 5]{Sound-Young}.
\subsection{Proof of Theorem \ref{main-thm}}
Recall the definition of $S(U_1,U_2)$ from (\ref{equ:13.1+}). 
Write $U=\frac{X}{(\log X)^{2^{50}}}$. Take $U_1=U_2= U$ and $Y=X^{\frac{1}{2}}U_1^{-\frac{1}{4}}U_2^{-\frac{1}{4}}$. 

Using these values, we can  simplify Lemmas \ref{lem:13.1}, \ref{lem:k=0}, \ref{lem:k=square} and \ref{lem:error}. In the following, we  give  the  detail of the simplification for Lemma \ref{lem:k=square}. The summation in Lemma \ref{lem:k=square} is 
\begin{align}
\sum_{j_1 + j_2 +j_3 +j_4 = 10} \frac{(-1)^{j_3} B(j_4)}{j_1 ! j_2 ! j_3 ! j_4 !} 
 (\log^{j_1} U)  (\log^{j_2} U) (\log^{j_3} X)  . 
 \label{equ:sum-B}
\end{align}
We consider the case $j_4 = 0$, and other cases can be done similarly. Assume $j_4 =0$ in (\ref{equ:sum-B}). Then by (\ref{equ:B(j)}), we have 
\begin{align*}
&\sum_{j_1 + j_2 +j_3  = 10} \frac{(-1)^{j_3} B(0)}{j_1 ! j_2 ! j_3 ! } 
 (\log^{j_1} U)  (\log^{j_2} U) (\log^{j_3} X) \\
 &= 
 26 \sum_{j_1 + j_2 +j_3  = 10} \frac{(-1)^{j_3} }{j_1 ! j_2 ! j_3 ! } 
 (\log^{j_1} U)  (\log^{j_2} U) (\log^{j_3} X) \\
 &= 26(\log^{10} X)   \sum_{j_1 + j_2 +j_3  = 10} \frac{(-1)^{j_3} }{j_1 ! j_2 ! j_3 ! } + O \left(\log^{9 + \varepsilon} X \right)\\
 &=\frac{26}{10!} \log^{10} X + O \left(\log^{9 + \varepsilon} X \right).
\end{align*}
The second last equality is due to $\log^j U = \log^{j} X + O(\log^{j-1+\varepsilon} X)$ for $j\geq 0$. The last equality is obtained by 
\[
\sum_{j_1 + j_2 +j_3  = 10} \frac{(-1)^{j_3} }{j_1 ! j_2 ! j_3 ! }  = \frac{1}{10!} \left. \frac{d^{10}}{dx^{10}} \right|_{x=0} (e^x e^x e^{-x} )= \frac{1}{10!}.
\]
Similarly, we can compute other cases in (\ref{equ:sum-B}). Combining all cases we can show (\ref{equ:sum-B}) is 
\begin{align*}
\left( \frac{26}{10!} - \frac{10}{9!} + \frac{1}{8!}\right)\log^{10} X + O \left(\log^{9 + \varepsilon} X \right) 
=
\frac{1}{2^4 \cdot 3^4 \cdot 5^2 \cdot 7}\log^{10} X + O \left(\log^{9 + \varepsilon} X \right).
\end{align*}
Using this fact,  Lemma \ref{lem:k=square}  can be simplified to 
\begin{align*}
S_1(k=\Box) 
  = - \frac{a_4 \tilde{\Phi}(1)}{2^5 \cdot 3^4 \cdot 5^2 \cdot 7 \cdot \pi^2}   X \log^{10} X
 +O \left(X \log^{9+\varepsilon} X  +  X (\log^{-20} X) \Phi_{(5)} \right).
\end{align*}

Now  by (\ref{equ:13.1+}), (\ref{equ:13.3}), combined with Lemmas \ref{lem:13.1}, \ref{lem:k=0}, \ref{lem:k=square} and \ref{lem:error}, we can  obtain that 
\begin{align}
S(U_1,U_2) &= \sideset{}{^*}\sum_{(d,2)=1} \left|A_{U}(\tfrac{1}{2};8d)\right|^2 \Phi\left(\tfrac{d}{X}\right) \nonumber\\
&=  \frac{a_{4} \tilde{\Phi}(1)}{2^6 \cdot 3^3 \cdot 5^2 \cdot 7 \cdot \pi^2} X \log^{10} X + O \left(X \log^{9+\varepsilon} X  +  X (\log^{-20} X) \Phi_{(5)} \right).
\label{equ:S(U,U)}
\end{align}

Define  $B_U(\frac{1}{2};8d) = L(\tfrac{1}{2},\chi_{8d})^2 - A_U(\frac{1}{2};8d)$. We claim that 
\begin{equation}
\sideset{}{^*}\sum_{(d,2)=1}|B_U(\tfrac{1}{2};8d)|^2 \Phi\left(\tfrac{d}{X}\right) \ll X \log^{9.5+\varepsilon} X .
\label{equ:B_U}
\end{equation}
In fact,  we  have
\begin{align*}
B_U(\tfrac{1}{2};8d)
=\frac{1}{\pi i} \int_{(c)} g(s) L(\tfrac{1}{2}+s,\chi_{8d})^2  \frac{(8d)^s-U^s}{s} ds.
\end{align*}
Since $\frac{(8d)^s-U^s}{s}$ is entire, we move the line of the integral to $\Re(s)=0$. By the bound $| \frac{(8d)^{it}-U^{it}}{it}| \ll  \log ( \frac{8d}{U} )$, $t\in \mathbb{R}$, we get  that
\[
B_U(\tfrac{1}{2};8d) \ll \log \left( \frac{8d}{U} \right) \int_{-\infty}^{\infty} |g(it)|  |L(\tfrac{1}{2}+it,\chi_{8d})|^2dt.
\]
This implies that the left-hand side of (\ref{equ:B_U}) is
\begin{align}
\ll  \left( \log \frac{X}{U} \right)^2 \int_{-\infty}^{\infty}\int_{-\infty}^{\infty}
 |g(it_1)||g(it_2)|\sideset{}{^*}\sum_{(d,2)=1}|L(\tfrac{1}{2}+it_1,\chi_{8d})|^2|L(\tfrac{1}{2}+it_2,\chi_{8d})|^2 \Phi(\tfrac{d}{X})dt_1dt_2.
 \label{equ:17.1}
\end{align}
Split the integral according to  whether $|t_1|,|t_2| \leq X$. If $|t_1|,|t_2| \leq X$, then use Theorem \ref{thm:3.1}. Otherwise, use Lemma \ref{lem:2.3}. This will establish (\ref{equ:B_U}).

Note that 
\begin{align*}
\sideset{}{^*}\sum_{(d,2)=1}L(\tfrac{1}{2},\chi_{8d})^4 \Phi\left(\tfrac{d}{X}\right)
&= \sideset{}{^*}\sum_{(d,2)=1}(A_U(\tfrac{1}{2};8d)+B_U(\tfrac{1}{2};8d))^2 \Phi\left(\tfrac{d}{X}\right)\\
&= \sideset{}{^*}\sum_{(d,2)=1}A_U(\tfrac{1}{2};8d)^2 \Phi\left(\tfrac{d}{X}\right) + \sideset{}{^*}\sum_{(d,2)=1}B_U(\tfrac{1}{2};8d)^2 \Phi\left(\tfrac{d}{X}\right)\\
&\quad + 2\sideset{}{^*}\sum_{(d,2)=1} A_U(\tfrac{1}{2};8d)
B_U(\tfrac{1}{2};8d)  \Phi\left(\tfrac{d}{X}\right).
\end{align*}
Using the Cauchy-Schwarz inequality on the third term, combined with (\ref{equ:S(U,U)}) and    (\ref{equ:B_U}), we obtain that
\begin{equation}
 \sideset{}{^*}\sum_{(d,2)=1}L(\tfrac{1}{2},\chi_{8d})^4\Phi(\tfrac{d}{X})  = 
  \frac{a_{4} \tilde{\Phi}(1)}{2^6 \cdot 3^3 \cdot 5^2 \cdot 7 \cdot \pi^2} X \log^{10} X 
  + O \left( X \log^{9.75+\varepsilon}X +  X (\log^{-5} X) \Phi_{(5)}\right).
  \label{equ:asy-smooth}
\end{equation}

In the following we remove the function $\Phi(\frac{d}{X})$ in the above summation.  Choose $\Phi$ such  that $\Phi(t) =1$ for all $t \in (1+Z^{-1}  ,  2-Z^{-1})$, $\Phi(t) =0$ for all $t \notin (1,2) $,  and  $\Phi^{(\nu)} (t)\ll_{\nu} Z^{\nu}$ for all $\nu \geq 0$. This implies that $\Phi_{(\nu)}  \ll_{\nu} Z^{\nu}$, and that $\tilde{\Phi}(1) = \check{\Phi} (0) = 1 + O(Z^{-1})$. Then by (\ref{equ:asy-smooth}), we get that 
\begin{align*}
 &\sideset{}{^*}\sum_{(d,2)=1}L(\tfrac{1}{2},\chi_{8d})^4\Phi(\tfrac{d}{X}) \\
 & = 
  \frac{a_{4}}{2^6 \cdot 3^3 \cdot 5^2 \cdot 7 \cdot \pi^2}X \log^{10} X  + O \left(  X (\log^{10} X) Z^{-1} +X \log^{9.75+\varepsilon}X +  X (\log^{-5} X) Z^5\right).
\end{align*}
Take $Z=\log X$. We have 
\begin{align}
 \sideset{}{^*}\sum_{\substack{X<d\leq 2X\\(d,2)=1}}L(\tfrac{1}{2},\chi_{8d})^4
 \geq \sideset{}{^*}\sum_{(d,2)=1}L(\tfrac{1}{2},\chi_{8d})^4\Phi(\tfrac{d}{X}) = 
  \frac{a_{4}}{2^6 \cdot 3^3 \cdot 5^2 \cdot 7 \cdot \pi^2}X \log^{10} X + O \left( X \log^{9.75+\varepsilon} X\right).
  \label{equ:lower}
\end{align}
Similarly,  we can choose  $\Phi(t)$ in (\ref{equ:asy-smooth})  such that $\Phi(t) = 1 $ for all $t \in [1 ,  2]$, $\Phi(t) =0$ for all $t \notin (1-Z^{-1},2+Z^{-1}) $,  and  $\Phi^{(\nu)}(t) \ll_{\nu} Z^{\nu}$ for all $\nu \geq 0$. Taking $Z=\log X$, we can deduce that 
\begin{align}
 \sideset{}{^*}\sum_{\substack{X<d\leq 2X\\(d,2)=1}}L(\tfrac{1}{2},\chi_{8d})^4
 \leq \sideset{}{^*}\sum_{(d,2)=1}L(\tfrac{1}{2},\chi_{8d})^4\Phi(\tfrac{d}{X}) = 
  \frac{a_{4}}{2^6 \cdot 3^3 \cdot 5^2 \cdot 7 \cdot \pi^2}X \log^{10} X + O \left( X \log^{9.75+\varepsilon} X\right).
  \label{equ:upper}
\end{align}
Combining (\ref{equ:lower}) and (\ref{equ:upper}), we obtain that
\begin{align*}
 \sideset{}{^*}\sum_{\substack{X<d\leq 2X\\(d,2)=1}}L(\tfrac{1}{2},\chi_{8d})^4
=  
  \frac{a_{4}}{2^6 \cdot 3^3 \cdot 5^2 \cdot 7 \cdot \pi^2}X \log^{10} X + O \left( X \log^{9.75+\varepsilon} X\right).
\end{align*}
Applying the above with $X= \frac{x}{2}$, $X= \frac{x}{4}$, $\dots$, we have proved Theorem \ref{main-thm}.
\subsection{Proof of Theorem \ref{main-thm1}} Write  $U = X^{1-4\varepsilon}$. By the Cauchy-Schwarz inequality, we obtain that 
\begin{align}
  \sideset{}{^*}\sum_{(d,2)=1}L (\tfrac{1}{2},\chi_{8d})^4\Phi \left(\tfrac{d}{X} \right) 
 \geq \frac{\left(\sideset{}{^*}\sum_{(d,2)=1} A_U(\frac{1}{2},8d) L (\tfrac{1}{2},\chi_{8d})^2  \Phi\left(\frac{d}{X} \right)\right)^2}
 {\sideset{}{^*}\sum_{(d,2)=1}\left(A_U(\tfrac{1}{2},8d)\right)^2 \Phi\left(\frac{d}{X}\right)}.
 \label{cauchy-lower}
\end{align}
Let $A^2$ and $B$ denote the numerator and denominator of the right-hand side in (\ref{cauchy-lower}), respectively. 

We first handle $B$. By (\ref{equ:13.1+})  and   (\ref{equ:13.3}), combined with Lemmas \ref{lem:13.1}, \ref{lem:k=0}, \ref{lem:k=square} and \ref{lem:error}, taking $Y=X^{\frac{1}{2}}U_1^{-\frac{1}{4}}U_2^{-\frac{1}{4}}$ and $ U_1=U_2= U$, we get that
\begin{align*}
B = S(U_1,U_2) = \frac{a_{4} \left(1-\frac{80}{3}\varepsilon + O(\varepsilon^2) \right)}{2^6 \cdot 3^3 \cdot 5^2 \cdot 7 \cdot \pi^2} \tilde{\Phi}(1) X \log^{10} X + O \left(X \log^{9} X + X \Phi_{(5)}\right),
\end{align*}
where the implied constant in $O(\varepsilon^2)$ is absolute.

For $A$, we have
\[
A = 4\sideset{}{^*}\sum_{(d,2)=1}\sum_{n_1=1}^{\infty}\sum_{n_2=1}^{\infty}\frac{\tau(n_1)\tau(n_2)\chi_{8d}(n_1n_2)}{\sqrt{n_1n_2}}h_1(d,n_1,n_2),
\]
where 
\[
h_1(x,y,z) := \Phi \left(\frac{x}{X}\right)\omega\left(\frac{y\pi}{U}\right)\omega\left(\frac{z\pi}{8x}\right).
\]
Note that the  difference between $A$ and $B$ lies in the difference between $h(x,y,z)$ and $h_1(x,y,z)$. By slightly modifying  the argument  for computing $B$, taking  $Y=X^{\frac{1}{2}}U^{-\frac{1}{4}} X^{-\frac{1}{4}}$, we can deduce that
\begin{align*}
A =  \frac{a_{4} \left(1-\frac{40}{3}\varepsilon + O(\varepsilon^2) \right)}{2^6 \cdot 3^3 \cdot 5^2 \cdot 7 \cdot \pi^2} \tilde{\Phi}(1) X \log^{10} X + O \left(X \log^{9} X+X \Phi_{(5)}\right),
\end{align*}
where the implied constant in $O(\varepsilon^2)$ is absolute.

Choose $\Phi$ such  that $\Phi(t) =1$ for all  $t \in (1+Z^{-1}  ,  2-Z^{-1})$, $\Phi(t) =0$ for all $t \notin (1,2) $,  and  $\Phi^{(\nu)} (t)\ll_{\nu} Z^{\nu}$ for all $\nu \geq 0$. Take $Z=\log X$.  Combining (\ref{cauchy-lower}) with the estimates for $A$ and $B$, we have 
\begin{align*}
  \sideset{}{^*}\sum_{\substack{(d,2)=1 \\ X <d \leq 2X}}L (\tfrac{1}{2},\chi_{8d})^4
 \geq  \left(1+ O(\varepsilon^2) \right)\frac{a_{4}}{2^6 \cdot 3^3 \cdot 5^2 \cdot 7 \cdot \pi^2}  X \log^{10} X .
\end{align*}
Having summed this with $X= \frac{x}{2}$, $X= \frac{x}{4}$, $\dots$, we  obtain  Theorem \ref{main-thm1}.

\section{Proof of Theorem \ref{thm:3.1}}\label{sec:upper bd}
In this section, we shall prove Theorem \ref{thm:3.1}. The proof here closely follows \cite[Section 6]{Sound-Young}.

Let $x\in \mathbb{R}$ with $x \geq10$,  and  $z \in \mathbb{C}$. Define
\begin{align*}
\mathcal{L}(z,x):=\left\{
 \begin{array}
  [c]{ll}
  \operatorname{log}\operatorname{log}x & |z|\leq(\operatorname{log}x)^{-1},\\
  -\operatorname{log}|z|& (\operatorname{log}x)^{-1}< |z|\leq 1,\\
  0&|z| >1.
 \end{array}
 \right.
\end{align*}
Let $z_1, z_2 \in \mathbb{C}$. We define
\begin{align*}
\mathcal{M}(z_1,z_2,x) :=\frac{1}{2} \left(\mathcal{L}(z_1,x)+\mathcal{L}(z_2,x) \right),
\end{align*}
and 
\begin{multline*}
\mathcal{V}  (z_1,z_2,x)\\
:=\frac{1}{2}\left(\mathcal{L}(2z_1,x)+\mathcal{L}(2z_2,x)+\mathcal{L}(2 \Re (z_1),x)+\mathcal{L}(2\Re (z_2),x)+2\mathcal{L}(z_1+z_2,x)+2\mathcal{L}(z_1+\overline{z_2},x)\right).
\end{multline*}
\begin{rem}
We see that the definition of $\mathcal{M}(z_1,z_2,x)$  is different from that in \cite[Section 6]{Sound-Young} by a factor $-1$, while $\mathcal{V}  (z_1,z_2,x)$ is the same.   The difference is due to the  different symmetry types of  families of $L$-functions
(see Katz-Sarnak \cite{Katz-Sarnak}). 
The family of quadratic Dirichlet $L$-functions is symplectic, whereas the  family of quadratic twists of a modular $L$-function  in  \cite{Sound-Young} is orthogonal.  For further explanation, we refer readers to \cite[p. 1111]{Sound-Young} and  \cite[p. 991]{Sound01}.
\end{rem}
\begin{prop}
Assume GRH for $L(s,\chi_d)$ for all fundamental discriminants $d$. Let $X$ be large. Let $z_1,z_2 \in \mathbb{C}$ with $0 \leq \Re (z_1), \Re (z_2) \leq \frac{1}{\log X}$, and $|\Im (z_1)|, |\Im (z_2)| \leq X$. Let $\mathcal{N}(V;z_1,z_2,X)$ denote the number of 
 fundamental discriminants $|d|\leq X$ such that
 \begin{align*}
  \operatorname{log}|L(\tfrac{1}{2}+z_1,\chi_{d})L(\tfrac{1}{2}+z_2,\chi_{d})|\geq 
  V +\mathcal{M}(z_1,z_2,X).
 \end{align*}
 Then for $10\sqrt{\operatorname{log}\operatorname{log} X}\leq V \leq \mathcal{V}(z_1,z_2,X)$,  we have
 \begin{align*}
  \mathcal{N}(V;z_1,z_2,X)\ll X\operatorname{exp}\left(-\frac{V^2}{2\mathcal{V}(z_1,z_2,X)}
  \left(1-\frac{25}{\operatorname{log}\operatorname{log}\operatorname{log}X}\right)\right);
 \end{align*}
 for $\mathcal{V}(z_1,z_2,X)< V \leq \frac{1}{16}\mathcal{V}(z_1,z_2,X)\operatorname{log}\operatorname{log}\operatorname{log}X$, 
 we have
 \begin{align*}
  \mathcal{N}(V;z_1,z_2,X)\ll X\operatorname{exp}\left(-\frac{V^2}{2\mathcal{V}(z_1,z_2,X)}
  \left(1-\frac{15V}{\mathcal{V}(z_1,z_2,X)\operatorname{log}\operatorname{log}\operatorname{log}X}\right)^2\right);
 \end{align*}
 finally, for $\frac{1}{16}\mathcal{V}(z_1,z_2,X)\operatorname{log}\operatorname{log}
 \operatorname{log}X<V$, we have
 \begin{align*}
  \mathcal{N}(V;z_1,z_2,X)\ll X\operatorname{exp}\left(-\frac{1}{1025}V\operatorname{log}V\right).
 \end{align*}
 \label{pro:3.2}
\end{prop}
\begin{proof}
It is helpful to keep in mind that $\log\log X + O(1) \leq \mathcal{V}(z_1,z_2,x) \leq 4\log\log X $. By slightly modifying the proof of the main  proposition in \cite{Sound01}, we obtain  that for any  $2 \leq x \leq X$,
\begin{align*}
\log |L(\tfrac{1}{2}+z_i,\chi_d)| & \leq \Re \left( \sum_{2 \leq n \leq x} \frac{ \Lambda (n) \chi_d (n)}{n^{\frac{1}{2}+\frac{\lambda_0}{\log x}+z_i} \log n} \frac{\log (\frac{x}{n})}{\log x} \right)+ (1+\lambda_0)\frac{\log X}{\log x} +O \left (\frac{1}{\log x} \right ), \  i=1,2,\\
\end{align*}
where $\lambda_0 = 0.56 \dots$ is the unique real number  satisfying $e^{-\lambda_0} = \lambda_0$.
It follows that 
\begin{align}
&\log  |L(\tfrac{1}{2}+z_1 , \chi_d)||L(\tfrac{1}{2}+z_2,\chi_d)|  \nonumber \\
 &\leq \Re \left( \sum_{\substack{p^l \leq x \\ l \geq 1}} \frac{\chi_d (p^l)}{l p^{l(\tfrac{1}{2}+\frac{\lambda_0}{\log x})}}  (p^{-lz_1} + p^{-lz_2}) \frac{\log (\frac{x}{p^l})}{\log x} \right)+ 2 (1+\lambda_0)\frac{\log X}{\log x} +O \left (\frac{1}{\log x} \right ).
\label{equ:3.1}
\end{align}

The terms with $l \geq 3$ in the the above sum contribute $O(1)$. Using the fact $\sum_{p|d} \frac{1}{p} \ll \log\log\log d$,  we get that
\begin{align}
&\Re  \left( \sum_{p^2 \leq x }  \frac{\chi_d (p^2)}{2 p^{1+\frac{2\lambda_0}{\log x}}}  (p^{-2z_1} + p^{-2z_2}) \frac{\log (\frac{x}{p^2})}{\log x} \right)\nonumber \\
&= \Re \left( \sum_{p \leq \sqrt{x} } \frac{1}{2 p^{1+\frac{2\lambda_0}{\log x}}}  (p^{-2z_1} + p^{-2z_2}) \frac{\log (\frac{x}{p^2})}{\log x} \right)  + O(\log\log\log X). 
\label{equ:3.2}
\end{align}
By RH, we can deduce  that 
\begin{align}
 \sum_{p\leq y}(p^{-2z_1}+p^{-2z_2})\log p = \frac{y^{1-2z_1}}{1-2z_1}
 +\frac{y^{1-2z_2}}{1-2z_2}+O \left(\sqrt{y} (\log Xy)^2 \right).
 \label{equ:3.2+}
\end{align}
The above sum also has a trivial bound $\ll y$. Combining (\ref{equ:3.2}) with these two bounds, by partial summation, we have
 \begin{align*}
  \sum_{p\leq \sqrt{x}}
    \frac{1}{2 p^{1+\frac{2\lambda_0}{\log x}}}(p^{-2z_1}+p^{-2z_2})
    \frac{\log (\frac{x}{p^2})}{\log x}=\mathcal{M}(z_1,z_2,x)+O (\log 
    \log \log X).
 \end{align*}
 Inserting above estimates into (\ref{equ:3.1}), by $\mathcal{M}(z_1,z_2,x) \leq \mathcal{M}(z_1,z_2,X)$, we obtain  that 
 \begin{multline}
 \log |L  (\tfrac{1}{2}+z_1 , \chi_d)||L(\tfrac{1}{2}+z_2,\chi_d)|  \\
    \leq \Re \left( \sum_{2<p \leq x} \frac{\chi_d (p)}{p^{\frac{1}{2}+\frac{\lambda_0}{\log x}}}  (p^{-z_1} +   p^{-z_2}) \frac{\log (\frac{x}{p})}{\log x} \right)
    +\mathcal{M}(z_1,z_2,X)+
    \frac{4 \log X}{\log x} + O(\log \log \log X).
 \label{equ:3.3}
 \end{multline}
 
 For brevity, put $\mathcal{V} := \mathcal{V} (z_1,z_2,X)$. Set
 \begin{align*}
 A:=\left\{
 \begin{array}
  [c]{ll}
  \frac{1}{2}\log\log\log X & 10 \sqrt{\log \log X}\leq V\leq \mathcal{V},\\
  \frac{\mathcal{V}}{2V}\log\log\log X & \mathcal{V}<V\leq \frac{1}{16}\mathcal{V}\log\log\log X,\\
  8 & V > \frac{1}{16}\mathcal{V}\log\log\log X.
 \end{array}
 \right.
\end{align*}

By taking $x = \log X$ in (\ref{equ:3.3}) and bounding the sum  over $p$ in (\ref{equ:3.3}) trivially, we know that $\mathcal{N} (V; z_1,z_2,X) =0$  for $V > \frac{5\log X}{\log \log X}  $. Thus, we can assume $V \leq   \frac{5\log X}{\log \log X}  $.

 From now on, we set $x=X^{A/V}$ and $z=x^{1/\log \log X}$. Let $S_1$ be the sum in (\ref{equ:3.3}) truncated to $p \leq z$, and $S_2$ be the sum over $z < p \leq x$. It follows from (\ref{equ:3.3}) that
\[
 \log |L (\tfrac{1}{2}+z_1 , \chi_d)||L(\tfrac{1}{2}+z_2,\chi_d)|
    \leq S_1 + S_2 + \mathcal{M}(z_1,z_2,X) + \frac{5V}{A}.
\]
Note that if $d$ satisfies  $\log |L (\frac{1}{2}+z_1 , \chi_d)||L(\tfrac{1}{2}+z_2,\chi_d)| \geq V +  \mathcal{M} (z_1,z_2,X)$, then either
\[
S_2 \geq \tfrac{V}{A}, \text{ or } S_1 \geq V_1 := V(1-\tfrac{6}{A}).
\]

Write 
\begin{align*}
\operatorname{meas}(X;S_1) &:= \# \{|d| \leq X \ : \ d \text{ is a fundamental discriminant, } S_1 \geq V_1 \},\\
\operatorname{meas}(X;S_2) &:= \# \{|d| \leq X \ : \ d \text{ is a fundamental discriminant, } S_2 \geq \tfrac{V}{A}  \}.
\end{align*}

For any $m \leq \frac{V}{2A} -1$, by \cite[Lemma 6.3]{Sound-Young}, we have
\begin{align*}
\sideset{}{^\flat}\sum_{|d|\leq X}|S_2|^{2m} \ll X\frac{(2m)!}{m!2^m}\left(\sum_{z<p\leq x}
\frac{4}{p}\right)^m \ll X (3m \log \log \log X)^m.
\end{align*}
By choosing $m = \lfloor \frac{V}{2A}\rfloor -1$, we get that
\begin{equation}
\operatorname{meas}(X;S_2) \ll X\operatorname{exp}\left(-\frac{V}{4A}\log V\right).
\label{equ:bd-S-2}
\end{equation}

We next estimate $\operatorname{meas}(X;S_1)$. For any $m \leq \frac{\frac{1}{2}\log X - \log \log X}{ \log z}$, by \cite[Lemma 6.3]{Sound-Young}, we obtain that
\begin{align}
\sideset{}{^\flat}\sum_{|d|\leq X}|S_1|^{2m} \ll X\frac{(2m)!}{m!2^m}\left(\sum_{p\leq z}\frac{|a(p)|^2}{p}\right)^m,
\label{equ:3.4}
\end{align}
where
\begin{align*}
 a(p)=\frac{\Re(p^{-z_1}+p^{-z_2})\log (\frac{x}{p})}{p^{\frac{\lambda_0}{\log x}}\log x}.
\end{align*}
By using (\ref{equ:3.2+}) and the partial summation, we can show that
\begin{align*}
\sum_{p\leq z}\frac{|a(p)|^2}{p}
\leq \frac{1}{4}\sum_{p\leq \sqrt{X}} \frac{1}{p}(p^{-z_1}+p^{-\overline{z_1}}+p^{-z_2}+p^{-\overline{z_2}})^2
=\mathcal{V}(z_1,z_2,X)+O(\log\log\log X).
\end{align*}
Together with (\ref{equ:3.4}), this yields
\begin{align*}
\operatorname{meas}(X ; S_1)
\ll X V_1^{-2m} \frac{(2m)!}{m!2^m} (\mathcal{V} + O(\log\log\log X))^m
\ll X\left(\frac{2m}{e} \cdot \frac{\mathcal{V}+O(\log\log\log X)}{V_1^2}\right)^m.
\end{align*}
Taking $m = \lfloor \frac{V_1^2} { 2\mathcal{V}} \rfloor$ when $V \leq \frac{(\log \log X)^2}{\log \log \log X}$, and taking $ m = \lfloor 10V \rfloor$ otherwise, we obtain that
\begin{align}
\operatorname{meas}(X ; S_1) \ll X\operatorname{exp}\left(-\frac{V_1^2}{2\mathcal{V}}\left(1+O\left(\frac{\log\log\log X}{\log\log X}\right)\right)\right)
+X\operatorname{exp}\left(-V\log V\right).
\label{equ:bd-S-1}
\end{align}

Using  the estimates (\ref{equ:bd-S-2}) and  (\ref{equ:bd-S-1}),  we can establish Proposition \ref{pro:3.2}. This completes the proof.
\end{proof}
For convenience,  in the following we show  a rough form of Proposition \ref{pro:3.2}. Let $k \in  \mathbb{R}_{>0}$ be fixed.
For $10\sqrt{\operatorname{log}\operatorname{log} X} \leq V \leq 4k \mathcal{V}(z_1,z_2,X)$, we have
\begin{align}
  \mathcal{N}(V;z_1,z_2,X)\ll X(\operatorname{log}X)^{o(1)}\operatorname{exp}\left(-\frac{V^2}{2\mathcal{V}(z_1,z_2,X)}\right),
  \label{equ:rough-01}
 \end{align}
and for $V > 4k \mathcal{V}(z_1,z_2,X)$, we have
 \begin{align}
  \mathcal{N}(V;z_1,z_2,X)\ll X(\operatorname{log}X)^{o(1)}\operatorname{exp}(-4kV).
  \label{equ:rough-02}
 \end{align}

Observe that
\begin{align*}
 \sideset{}{^\flat}\sum_{|d|\leq X}   |L (\tfrac{1}{2}+z_1 , \chi_d)L(\tfrac{1}{2}+z_2,\chi_d)|^k
  &=-\int_{-\infty}^{\infty}\operatorname{exp}(kV+k\mathcal{M}(z_1,z_2,X))d \mathcal{N}(V;z_1,z_2,X) \nonumber\\
  &=k\int_{-\infty}^{\infty}\operatorname{exp}(kV+k\mathcal{M}(z_1,z_2,X))\mathcal{N}(V;z_1,z_2,X)dV.
\end{align*}
Inserting the rough bounds (\ref{equ:rough-01}) and (\ref{equ:rough-02}) into the integral above, we can deduce that
\begin{thm}
Assume GRH for $L(s,\chi_d)$ for all fundamental discriminants $d$. Let $X$ be large. Let $z_1,z_2 \in \mathbb{C}$ with $0 \leq \Re (z_1), \Re (z_2) \leq \frac{1}{\log X}$, and $|\Im (z_1)|, |\Im (z_2)| \leq X$. Then for any positive real number $k$ and any $\varepsilon>0$, we have
 \begin{align*}
  \sideset{}{^\flat}\sum_{|d|\leq X} |L(\tfrac{1}{2}+z_1,\chi_{d})
 L(\tfrac{1}{2}+z_2,\chi_{d})|^{k}
  \ll_{k,\varepsilon}X(\operatorname{log}X)^\varepsilon
  \operatorname{exp}\left(k \mathcal{M}(z_1,z_2,X)+\frac{k^2}{2}\mathcal{V}(z_1,z_2,X)\right).
 \end{align*}
\label{thm:3.3}
\end{thm}

In the rest  of this section,  we complete the proof of Theorem \ref{thm:3.1}.
\begin{proof}[Proof of Theorem \ref{thm:3.1}]
By Theorem \ref{thm:3.3} and the fact that $\mathcal{L}(z,x) \leq \log \log x$ for $z \in \mathbb{C}, x \geq  10$,  we can trivially  get that
\begin{equation}
\sideset{}{^\flat}\sum_{|d|\leq X}  |L(\tfrac{1}{2} + z_1,\chi_d)|^k|L(\tfrac{1}{2} + z_2,\chi_d)|^k 
 \ll_{k,\varepsilon} X(\operatorname{log}X)^{2k^2 +k+\varepsilon}.
 \label{equ:trivial-upper-L}
 \end{equation}

Now we assume $|\Im (z_1) - \Im (z_2)| \geq \frac{1}{\log X}$. Write $t_1 = \Im (z_1)$ and $t_2 = \Im (z_2)$. 

If $t_1t_2 \geq 0$, then $|t_1-t_2| \leq |t_1+t_2|  \leq \max (2|t_1|, 2|t_2|)$,  say $|t_1+t_2| \leq 2|t_1|$. Note that $\mathcal{L}(y,X$) is a decreasing function  for $y \geq 0 $. Thus, we have 
\begin{align*}
\mathcal{L}(z_1,X),\ \mathcal{L}(2z_1,X),\ \mathcal{L}(z_1+z_2,X), \ \mathcal{L}(z_1+\overline{z_2},X)& \leq \mathcal{L}(|t_1-t_2|,X) + O(1) \\ &\leq \max (0, -\log |t_1-t_2|) +O(1).
\end{align*}
This together with 
\[
\mathcal{L}(z_2,X),\   \mathcal{L}(2z_2,X),\  \mathcal{L}(2 \Re(z_1),X),\  \mathcal{L}(2 \Re(z_2),X) \leq \log\log X 
\]
implies
\begin{align}
2\mathcal{M}(z_1,z_2,X) + 2\mathcal{V}(z_1,z_2,X) \leq 4\log\log X + \max \{0,- 6\log |t_1-t_2|\}+O(1).
\label{equ:3.5}
\end{align}

On the other hand, if $t_1t_2 < 0$, then $|t_1-t_2| = |t_1|+|t_2|\leq \max \{|2t_1|,|2t_2|\}$, say $|t_1-t_2| \leq |2t_2|$. It implies that $|t_1| \leq |t_2|$ and that $\mathcal{L}(2t_2,X)\leq \mathcal{L}(|t_1-t_2|,X)$. Note $|t_1-t_2|=2|t_1|+|t_1+t_2|$, so $|t_1-t_2| \leq \max \{4|t_1|, 2|t_1+t_2|\}$. In fact, if $|t_1-t_2| > 4 |t_1|$, then $2|t_1|+|t_1+t_2| > 4|t_1|$, which implies $|t_1| \leq \frac{1}{2}  |t_1+t_2| $. It means  $|t_1-t_2|=2|t_1|+|t_1+t_2| \leq 2 |t_1 +t_2|$. Without loss of generality, we can say $|t_1-t_2| \leq 4|t_1|$. It follows that $\mathcal{L}(z_1,X), \mathcal{L}(2z_1,X)\leq \mathcal{L}(|t_1-t_2|,X)+O(1)$. Now we have 
\begin{align*}
\mathcal{L}(z_1,X),\ \mathcal{L}(2z_1,X),\ \mathcal{L}(z_2 ,X), \ \mathcal{L}(2z_2 ,X), \ \mathcal{L}(z_1+\overline{z_2},X)& \leq \mathcal{L}(|t_1-t_2|,X) + O(1) \\ &\leq \max (0, -\log |t_1-t_2|) +O(1).
\end{align*}
This combined with 
\[
\mathcal{L}(2 \Re(z_1),X),\  \mathcal{L}(2 \Re(z_2),X),\  \mathcal{L}(z_1+z_2,X) \leq \log\log X 
\]
also implies (\ref{equ:3.5}). 

By inserting  (\ref{equ:3.5}) into Theorem \ref{thm:3.3}, we can show for $|\Im (z_1) - \Im (z_2)| \geq \frac{1}{\log X}$,
\begin{equation}
\sideset{}{^\flat}\sum_{|d|\leq X}  |L(\tfrac{1}{2} + z_1,\chi_d)|^2|L(\tfrac{1}{2} + z_2,\chi_d)|^2 
 \ll X(\operatorname{log}X)^{4+\varepsilon} \left(1 + \frac{1}{|t_1 - t_2|^6} \right).
 \label{equ:nontrivial-upper-L}
\end{equation}

By combining (\ref{equ:nontrivial-upper-L}) and (\ref{equ:trivial-upper-L}) with $k=2$, we have proved Theorem \ref{thm:3.1}. 
\end{proof}

\noindent \textbf{Acknowledgements.} I would like to thank my supervisors  Habiba Kadiri and  Nathan Ng, for suggesting this problem to me, and for having numerous helpful discussions. I  would also like to thank Matilde Lal{\'\i}n,  Keiju Sono and  Peng-jie Wong for  their valuable comments. Lastly, I would like to thank the referee for their extensive feedback and constructive comments.

 \bibliographystyle{plain}
 \bibliography{Revised-the-fourth-moments-of-Dirichlet-L-functions-brief-version}
\end{document}